\documentclass[12pt,leqno]{article}
\usepackage{hyperref}
\usepackage{soul}
\usepackage{appendix}
\usepackage{pifont}
\usepackage{rotating}
\usepackage[doc]{optional}
\usepackage{xcolor}
\definecolor{labelkey}{rgb}{0,0.08,0.45}
\definecolor{rekey}{rgb}{0,0.6,0.0}
\definecolor{Brown}{rgb}{0.45,0.0,0.05}
\usepackage{exscale,relsize}
\usepackage{amsmath}
\usepackage{amsfonts}
\usepackage{amssymb}
\usepackage{calc}
\usepackage{theorem}
\usepackage{pifont}      %needed by dingautolist
\usepackage{graphicx}
\DeclareMathOperator{\weakstarly}{\rightharpoondown_{\mathrm{w*}}}

\newcommand{\wk}{\ensuremath{\operatorname{w*}}}

\newcommand{\scal}[2]{\langle{{#1},{#2}}\rangle}

\newcommand{\RR}{\ensuremath{\mathbb R}}

\newcommand{\RX}{\ensuremath{\,\left]-\infty,+\infty\right]}}
\newcommand{\RXX}{\ensuremath{\,\left[-\infty,+\infty\right]}}

\newcommand{\NN}{\ensuremath{\mathbb N}}

\newcommand{\To}{\ensuremath{\rightrightarrows}}

\newcommand{\spand}{\operatorname{span}}

\newcommand{\dom}{\ensuremath{\operatorname{dom}}}

\newcommand{\epi}{\ensuremath{\operatorname{epi}}}

\newcommand{\inte}{\ensuremath{\operatorname{int}}}

\newcommand{\menge}[2]{\big\{{#1} \mid {#2}\big\}}

\renewcommand{\phi}{\ensuremath{\varphi}}

\newtheorem{theorem}{Theorem}[section]
\newtheorem{lemma}[theorem]{Lemma}
\newtheorem{fact}[theorem]{Fact}
\newtheorem{corollary}[theorem]{Corollary}
\newtheorem{proposition}[theorem]{Proposition}

\theoremstyle{plain}{\theorembodyfont{\rmfamily}
}
\theoremstyle{plain}{\theorembodyfont{\rmfamily}
}
\theoremstyle{plain}{\theorembodyfont{\rmfamily}
}
\theoremstyle{plain}{\theorembodyfont{\rmfamily}
\newtheorem{example}[theorem]{Example}}
\theoremstyle{plain}{\theorembodyfont{\rmfamily}
\newtheorem{remark}[theorem]{Remark}}

\theoremstyle{plain}{\theorembodyfont{\rmfamily}
}

\begin{document}

%\sffamily

\title{\sffamily{Conditions for zero duality gap in convex programming }}

\author{
Jonathan M.
Borwein\thanks{CARMA, University of Newcastle, Newcastle, New South
Wales 2308, Australia. E-mail:
\texttt{jonathan.borwein@newcastle.edu.au}.
Laureate Professor at the University of Newcastle and Distinguished Professor at  King
Abdul-Aziz University, Jeddah. },\;
Regina S. Burachik\thanks{School of Information Technology and  Mathematical Sciences,
 University of South Australia, Mawson Lakes, SA 5095, Australia. E-mail:
\texttt{regina.burachik@unisa.edu.au}.},\;
and Liangjin\
Yao\thanks{CARMA, University of Newcastle,
 Newcastle, New South Wales 2308, Australia.
E-mail:  \texttt{liangjin.yao@newcastle.edu.au}.}}

\date{April 15, revision, 2013}
\maketitle
\begin{abstract}
We introduce and study a new dual condition which characterizes zero
duality gap in nonsmooth convex optimization. We prove that our
condition is less restrictive than all existing constraint qualifications,
including the closed epigraph condition. Our dual condition was inspired by, and is less restrictive than, the
so-called Bertsekas' condition for monotropic programming problems. We
give several corollaries of our result and special cases as
applications. We pay special attention to the polyhedral and sublinear
cases, and their implications in convex optimization.
\end{abstract}

\noindent {\bfseries 2010 Mathematics Subject Classification:}\\
{Primary  49J52, 48N15;
Secondary
90C25, 90C30, 90C46}
\noindent

\noindent {\bfseries Keywords:}
Bertsekas Constraint Qualification,
Fenchel conjugate, Fenchel duality theorem,
 normal cone operator,
 inf-convolution,
$\varepsilon-$subdifferential operator,
subdifferential operator,
zero duality gap.

\noindent

\section{Introduction}

Duality theory establishes an interplay between an optimization
problem, called the {\em primal}, and another optimization problem, called
the {\em dual}. A main target  of this approach is the establishment of the so-called {\em
  zero duality gap}, which means that the optimal values of primal and
dual problems coincide. Not all convex problems enjoy the zero duality
gap property, and this has motivated the quest for assumptions on the
primal problem which ensure zero duality gap (see \cite{Tseng} and
references therein).

Recently Bertsekas considered such an assumption
for a specific convex optimization problem, called the {\em extended
  monotropic programming problem}, the origin of which goes back to
Rockafellar (see \cite{Rock81,Rock84}). Following Bo\c{t} and Csetnek
\cite{BotCse2}, we study this problem in the following setting. Let
$\{X_i\}^m_{i=1}$ be separated locally convex spaces and let
$f_i:X_i\rightarrow\RX$ be proper lower semicontinuous and convex for
every $i\in\{1,2,\ldots,m\}$. Consider the minimization problem
\begin{align*}
(P)\qquad
&p:=\inf \left(\sum_{i=1}^m f_i(x_i)\right)\quad
\text{subject to}\, (x_1,\ldots, x_m)\in S,
\end{align*}
where $S\subseteq X_1\times X_2\times\cdots X_m$ is a linear closed
subspace. The dual problem is given as follows:
\begin{align*}
(D)\qquad
d:=\sup \left(\sum_{i=1}^m -f^*_i(x_i^*)\right)\quad \text{subject to}\,(x^*_1,\ldots, x^*_m)\in S^{\bot}.
\end{align*}
We note that formulation $(P)$ includes any general convex optimization
problem. Indeed, for $X$ a separated locally convex space, and $f:X\rightarrow\RX$ a proper lower semicontinuous and convex function, consider the problem
\begin{align*}
(CP)\qquad
&\inf f(x)\quad
\text{subject to}\, x\in C,
\end{align*}
where $C$ is a closed and convex set. Problem $(CP)$ can be reformulated
as
\begin{align*}
\qquad
&\inf \left\{f(x_1)+\iota_C(x_2)\right\}\quad
\text{subject to}\, (x_1, x_2)\in S=\{(y_1,y_2)\in X\times X\::\: y_1=y_2\},
\end{align*}
where $\iota_C$ is the indicator function of $C$.

Denote by $v(P)$ and $v(D)$ the optimal values of $(P)$ and $(D)$,
respectively. In the finite dimensional setting, Bertsekas proved in
\cite[Proposition~4.1]{BerT} that a \emph{zero duality gap} holds for
problems $(P)$ and $(D)$ (i.e., $p=v(P)=v(D)=d$) under the following
condition:
\begin{align*}
\qquad&N_{S}(x)+\Big(\partial_{\varepsilon} f_1(x_1),\ldots,\partial_{\varepsilon} f_m(x_m)\Big)\quad\text{is closed} \\
&\quad\text{for every $\varepsilon>0$, $(x_1,\ldots, x_m)\in S$ and $x_i\in\dom f_i,\,\forall i\in\{1,2,\ldots,m\}$},
\end{align*}
where the sets $\partial_{\varepsilon} f_i(x_i)$ are the
epsilon-subdifferentials of the $f_i$ at $x_i$ (see
\eqref{def:eps-sub} for the definition).  In \cite[Theorem~3.2]{BotCse2}, Bo\c{t} and
Csetnek extended this result to the setting of separated locally
convex spaces.

\noindent Burachik and Majeed \cite{BurMaj} presented a zero duality gap
property for a monotropic programming problem in which the subspace
constraint $S$ in $(P)$ is replaced by a closed cone $C$, and the
orthogonal subspace $S^{\perp}$ in $(D)$ is replaced by the dual cone
$C^*:=\{x^*\mid \inf_{c\in C}\langle x^*, C\rangle\geq0\}$.
Defining $g_i:X_1\times X_2\times \cdots \times X_m\rightarrow\RX$ by  $g_i(x_1,\ldots, x_m):=f_i(x_i)$, we have
\begin{align*}
 (P)\qquad
 p&=\inf \left(\sum_{i=1}^m f_i(x_i)\right)\quad
 \text{subject to}\, (x_1,\ldots, x_m)\in C\\
 &=\inf\Big(\iota_C(x)+\sum_{i=1}^m g_i(x)\Big)\\
 (D)\qquad d&=
 \sup_{(x^*_1,\ldots, x^*_m)\in C^*} \sum_{i=1}^m -f^*_i(x_i^*),
 \end{align*}
where  $C\subseteq X_1\times X_2\times\cdots\times X_m$ is a closed convex cone. In \cite[Theorem~3.6]{BurMaj}, Burachik and Majeed proved that
\begin{equation}\label{BCQ}
\hspace{-1.2cm}\hbox{
if }\partial_{\epsilon}\iota_C (x)+\sum_{i=1}^m \partial_{\varepsilon} g_i(x) \hbox{ is weak$^*$ closed
 for every } x\in C\cap \big(\bigcap_{i=1}^m\dom  g_i\big),
\end{equation}
then $v(p)=v(D)$.
Note that $\partial_{\epsilon}\iota_C (x)+\sum_{i=1}^m
\partial_{\varepsilon} g_i(x)= \partial_{\epsilon}\iota_C
(x)+\Big(\partial_{\varepsilon} f_1(x_1),\ldots,\partial_{\varepsilon}
f_m(x_m)\Big)$.  Thence, Burachik and Majeed's result extends Bo\c{t}
and Csetnek's result and Bertsekas' result to the case of cone
constraints. From now on, we focus on a more general form of condition \eqref{BCQ}, namely
\begin{equation}\label{BCQ1}
\sum_{i=1}^m \partial_{\varepsilon} f_i(x) \hbox{ is weak$^*$ closed},
\end{equation}
where $f_i:X\rightarrow\RX$ is a proper lower semicontinuous and convex function for all $i=1,\ldots,m$. We will refer to \eqref{BCQ1} as {\em the Bertsekas Constraint Qualification}.

In none of these results, however,  is there a direct link
between \eqref{BCQ1} and the zero duality gap
property. One of the aims of this paper is to establish such a link precisely.

Another constraint qualification is the so-called {\em closed epigraph
condition}, which was first introduced by Burachik and Jeyakumar in
\cite[Theorem 1]{BurJeya} (see also \cite{BotWank1, LiNg}). This condition is stated as
\begin{equation}\label{CEC}
\epi f^*_1+\cdots+\epi f^*_m \hbox{ is weak$^*$ closed  in the topology $\omega (X^*, X)\times \RR$}.
\end{equation}

Condition \eqref{CEC} does not imply \eqref{BCQ1}. This was
recently shown in \cite[Example 3.1]{BurMaj}, in which \eqref{BCQ1}
(and hence zero duality gap) holds, while \eqref{CEC} does not.

We recall from \cite[Proposition 6.7.3]{Laurent} the following
characterization of the zero duality gap property for $(P)$ and $(D)$,
which uses the {\em infimal convolution} (see \eqref{def:infconv} for
its definition) of the conjugate functions $f^*_i$.
\begin{align*}
(P)\qquad
p&=\inf \left(\sum_{i=1}^m f_i(x)\right)=-\big(\sum_{i=1}^m f_i\big)^*(0)\\
(D)\qquad d&=-\left(f^*_1\Box\cdots \Box f^*_m\right)(0).
\end{align*}
Hence, zero duality gap is  tantamount to the equality
\[
\big(\sum_{i=1}^m
f_i\big)^*(0)=\big(f^*_1\Box\cdots \Box f^*_m\big)(0).
\]

In our main result (Theorem~\ref{Theman:1} below), we introduce a new closedness property, stated as follows.
There exists $K>0$ such for every $x\in \bigcap_{i=1}^{m}\dom f_i$ and every $\varepsilon>0$,
\begin{equation}\label{LiangjinCQ}
\overline{\left[\sum_{i=1}^{m}\partial_\varepsilon
    f_i(x)\right]}^{\wk} \subseteq
\sum_{i=1}^{m}\partial_{K\varepsilon} f_i(x).
\end{equation}
 Theorem \ref{Theman:1} below proves that this
property is equivalent to
\begin{equation}\label{Equality Condition}
\big(\sum_{i=1}^m
f_i\big)^*(x^*)=\big(f^*_1\Box\cdots \Box f^*_m\big)(x^*), \hbox{ for all } x^*\in X^*.
\end{equation}
Condition \eqref{LiangjinCQ} is easily implied by \eqref{BCQ}, since
the latter implies that \eqref{LiangjinCQ} is true for the choice
$K=1$. Hence, Theorem \ref{Theman:1} shows exactly how and why
\eqref{BCQ} implies a zero duality gap. Moreover, in view of
\cite[Theorem 1]{BurJeya}, we see that our new condition
\eqref{LiangjinCQ} is strictly less restrictive than the closed
epigraph condition. Indeed, the latter implies not only
\eqref{Equality Condition} but also exactness of the infimal
convolution everywhere. Condition\eqref{Equality Condition} with
exactness is equivalent to \eqref{CEC}.  Condition \eqref{CEC}, in turn, is less
restrictive than the interiority-type conditions.

In the present paper, we focus
on the following kind of interiority condition:
\begin{equation}\label{Interiority}
\dom f_1\cap\big(\bigcap_{i=2}^{m}\inte\dom f_i\big)\neq\varnothing.
\end{equation}

In summary, we have
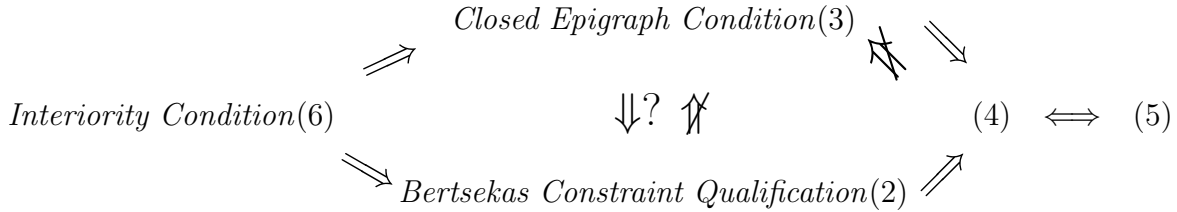
\[
\begin{tabular}{ccccccc}
 && {\em Closed Epigraph Condition}\eqref{CEC}&\hspace{-.9cm}\vspace{.2cm}\begin{rotate}{225}
\LARGE{$\not\!\,\,\Downarrow$}
\end{rotate}\hspace{.2cm} \begin{rotate}{-45}
\large{$\Longrightarrow$}
\end{rotate}  &&&\\
  &\begin{rotate}{25}
\large{$\Longrightarrow$}
\end{rotate}& & &&\\
 {\em Interiority Condition}\eqref{Interiority}&&\hspace{.2cm}\Large{$\Downarrow$?} $\not\!\,\Uparrow$& &\eqref{LiangjinCQ}&$\Longleftrightarrow$&\eqref{Equality Condition}\\
  &\hspace{-.8cm}\begin{rotate}{-30}
\large{$\Longrightarrow$}
\end{rotate}& & &&\\
&&{\em Bertsekas Constraint Qualification}\eqref{BCQ1}&\hspace{-.3cm}\begin{rotate}{45}
\large{$\Longrightarrow$}
\end{rotate}  & &&
\end{tabular}
\]
Example 3.1 in \cite{BurMaj} allows us to assert that the Bertsekas
Constraint Qualification is not more restrictive than the Closed
Epigraph Condition. This example also shows that our condition
\eqref{LiangjinCQ} does not imply the closed epigraph condition. It is
still an open question whether a more precise relationship can be
established between the closed epigraph condition and Bertsekas
Constraint Qualification. The arrow linking \eqref{Interiority} to
\eqref{CEC} has been established by Z\u{a}linescu in
\cite{Zalinescu1,Zalinescu}. All other arrows are, as far as we know,
new, and are established by us in this paper. Some clarification is in
order regarding the arrow from \eqref{Interiority} to the {\em
  Bertsekas Constraint Qualification}\eqref{BCQ1}. It is clear that
for every $x_0\in \dom f_1\cap\big(\bigcap_{i=2}^{m}\inte\dom
f_i\big)$, the set $\sum_{i=1}^m \partial_{\varepsilon} f_i(x_0)
\hbox{ is weak$^*$ closed}$. Indeed, this is true because the latter
set is the sum of a weak$^*$ compact set and a weak$^*$ closed
set. Our Lemma \ref{DualCorL:Le4} establishes that, under assumption
\eqref{Interiority}, the set $\sum_{i=1}^{m}\partial_{\varepsilon}
f_i(x)$ is weak$^*$ closed {\em for every point} $x\in
\big(\bigcap_{i=1}^{m}\dom f_i\big)$.

A well-known result, which is not easily found in the literature, is
the equivalence between \eqref{CEC} and the equality \eqref{Equality
  Condition} with exactness of the infimal convolution everywehere in
$X^*$. For convenience and possible future use, we have included the
proof of this equivalence in the present paper (see Proposition \ref{35}).

The layout of our paper is as follows. The next section contains the
necessary preliminary material. Section \ref{sec:main} contains our
main result, and gives its relation with the Bertsekas Constraint Qualification
\eqref{BCQ1}, with the closed epigraph condition \eqref{CEC}, and with
the interiority conditions \eqref{Interiority}. Still in this section
we establish stronger results for the important special case in which all $f_i$s
are sublinear. We finish this section by showing that our closedness condition
allows for a simplification of the well-known Hiriart-Urruty and Phelps formula
for the subdifferential of the sum of convex functions. In Section~\ref{intepCorL} we show that (generalized) interiority
conditions imply \eqref{BCQ1}, as well as \eqref{CEC}. We also provide
some additional consequences of Corollary~\ref{DualCorL:RL4}, including
various forms of Rockafellar's Fenchel duality result. At the end of Section~\ref{intepCorL} we establish stronger
results for the case involving polyhedral functions. We end the paper with some conclusions and open questions.

\section{Preliminaries}\label{sec:pre}

Let $I$ be a directed set with a partial order $\preceq$. A subset $J$
of $I$ is said to be {\em terminal} if there exists $j_0\in I$ such
that every successor $k\succeq j_0$ verifies $k\in J$.  We say that a net
$\{s_{\alpha}\}_{\alpha\in I}\subseteq \RR$ is {\em eventually bounded} if there exists
a terminal set $J$ and $R>0$ such that $|s_{\alpha}|\le R$ for every $\alpha \in J$.

We assume throughout that $X$ is a separated (i.e., Hausdorff) locally
convex topological vector space and $X^*$ is its continuous
dual  endowed with the weak$^*$ topology $\omega(X^*,X)$. Given a subset $C$ of $X$, $\inte C$ is the \emph{interior} of
$C$. We next recall standard notions from convex analysis, which can be found, e.g., in
\cite{BC2011,BorVan,BurIus, ph, Rock70CA,RockWets,Zalinescu}.  For the
set $D\subseteq X^*$, $\overline{D}^{\wk}$ is the weak$^{*}$ closure
of $D$.  The \emph{indicator function} of $C$, written as $\iota_C$,
is defined at $x\in X$ by
\begin{align}
\iota_C (x):=\begin{cases}0,\,&\text{if $x\in C$;}\\
+\infty,\,&\text{otherwise}.\end{cases}\end{align}
The \emph{normal cone} operator of $C$ at $x$ is defined by
$N_C(x):= \menge{x^*\in X^*}{\sup_{c\in C}\scal{c-x}{x^*}\leq 0}$, if
$x\in C$; and $N_C(x):=\varnothing$, if $x\notin C$.
%If $C$ is a cone, we define $C^*$ by $C^*:=\{x^*\mid \inf_{c\in C}\langle x^*, C\rangle\geq0\}$.
If $S\subseteq X$ is a subspace, we define $S^\bot$ by $S^\bot := \{z^*\in X^*\mid\langle
z^*,s\rangle= 0,\quad \forall s\in S\}$.
Let $f\colon X\to \left[-\infty, +\infty\right]$. Then
$\dom f:= f^{-1}\left[-\infty, +\infty\right[$ is the \emph{domain} (or \emph{effective domain}) of $f$, and
$f^*\colon X^*\to\RXX\colon x^*\mapsto
\sup_{x\in X}\{\scal{x}{x^*}-f(x)\}$ is
the \emph{Fenchel conjugate} of $f$.
The \emph{epigraph} of $f$ is $\epi f := \menge{(x,r)\in
X\times\RR}{f(x)\leq r}$.

The  \emph{lower semicontinuous hull} of $f$ is denoted by $\overline{f}$.
We say $f$ is proper if $\dom f\neq\varnothing$ and $f>-\infty$.
Given a  function $f$, the \emph{subdifferential} of
$f$ is the point-to-set mapping $\partial f\colon X\To X^*$ defined by
 $$\partial f(x):= \begin{cases}\{x^*\in
X^*\mid(\forall y\in X)\; \scal{y-x}{x^*} + f(x)\leq
f(y)\}\,&\text{if}\, f(x)\in\RR;\\
\varnothing\,&\text{otherwise}.
\end{cases}$$
Given $\varepsilon\ge 0$, the \emph{$\varepsilon-$subdifferential} of
$f$ is the point-to-set mapping $\partial_{\varepsilon} f\colon X\To X^*$ defined by
\begin{equation}\label{def:eps-sub}
\partial_{\varepsilon} f(x):=\begin{cases}\{x^*\in
X^*\mid(\forall y\in X)\; \scal{y-x}{x^*} + f(x)\leq
f(y)+\varepsilon\}\,&\text{if}\, f(x)\in\RR;\\
\varnothing\,&\text{otherwise}.
\end{cases}
\end{equation}
Thus, if $f$ is not proper, then $\partial_{\varepsilon} f(x)
=\varnothing$ for every $\varepsilon\geq0$ and $x\in X$.
Note also that if $f$ is convex and there exists $x_0\in X$ such that
$\overline{f}(x_0)=-\infty$, then $\overline{f}(x)=-\infty,\forall x\in\dom \overline{f}$ (see
\cite[Proposition~2.4]{EkeTem} or \cite[page~867]{BanLoZa}).

Let $f: X\rightarrow\RX$. We say $f$ is a \emph{sublinear} function
if $f(x+y)\leq f(x)+f(y)$, $f(0)=0$,  and $f(tx)=tf(x)$ for every $x,y\in\dom f$ and $t\geq0$.

Let $Z$ be a separated locally convex  space and let $m\in\NN$. For
a family of functions $\psi_1,\ldots, \psi_m$ such that
$\psi_i:Z\rightarrow\left[-\infty,+\infty\right]$ for all $i=1,\ldots,m$, we define its \emph{infimal convolution} as the function
$(\psi_1\Box\cdots \Box \psi_m):Z\to \RXX$ as
\begin{equation}\label{def:infconv}
\big(\psi_1\Box\cdots \Box \psi_m\big)z=\inf_{\sum_{i=1}^m z_i =z}\big\{\psi(z_1)+\cdots +\psi_m(z_m)\big\}.
\end{equation}

We denote by $\weakstarly$ the weak$^*$ convergence of nets in $X^*$.

\section{Our main results}\label{sec:main}

The following formula will be important in the proof of our main result.

\begin{fact}\label{BotCsCo}\emph{(See \cite[Corollary~2.6.7]{Zalinescu} or \cite[Theorem~3.1]{BotCse2}.)}
 Let $f,g:X\rightarrow\RX$ be  proper lower semicontinuous and convex. Then for every $x\in X$ and $
 \varepsilon\geq0$,
  \begin{align*}
 \partial_{\varepsilon}(f+g)(x)= \bigcap_{\eta>0}\overline{ \left[ \bigcup_{\varepsilon_1\geq0, \varepsilon_2\geq0,\, \varepsilon_1+\varepsilon_2=\varepsilon+\eta}
 \big(\partial_{\varepsilon_1} f(x)+\partial_{\varepsilon_2} g(x)\big)\right]
 }^{\wk}.
 \end{align*}
\end{fact}

 We now come to our main result. The proof in part follows that of \cite[Theorem~3.2]{BotCse2}.

%principal

\begin{theorem}\label{Theman:1}
Let $m\in\NN$, and
 $f_i:X\rightarrow\RX$ be proper convex with $\bigcap_{i=1}^{m}\dom f_i\neq\varnothing$,
 where $i\in\{1,2,\ldots,m\}$. Suppose  that $\overline {f_i}=f_i$  on $\bigcap_{i=1}^{m}\dom\overline {f_i}$.
 Then the following four conditions are equivalent.
 \begin{enumerate}
 \item \label{mainGed:1}
There exists $K>0$ such that for every $x\in \bigcap_{i=1}^{m}\dom f_i$, and every $\varepsilon>0$,
\begin{align*}
\overline{\left[\sum_{i=1}^{m}\partial_\varepsilon f_i(x)\right]}^{\wk}\subseteq\sum_{i=1}^{m}\partial_{K\varepsilon} f_i(x).
\end{align*}

\item \label{mainGed:2}$\left(\sum_{i=1}^m f_i\right)^*=f_1^*\Box\cdots \Box f_m^*$
in $X^*$.

\item \label{mainGed:3} $f_1^*\Box\cdots \Box f_m^*$ is weak$^*$ lower semicontinuous.

 \item  \label{mainGed:4} For every $x\in X$ and $
 \varepsilon\geq0$,
  \begin{align*}
 \partial_{\varepsilon}(f_1+\cdots+f_m)(x)= \bigcap_{\eta>0} \left[ \bigcup_{\varepsilon_i\geq0, \, \sum_{i=1}^m\varepsilon_i=\varepsilon+\eta}
 \Big(\partial_{\varepsilon_1} f_1(x)+\cdots+\partial_{\varepsilon_m} f_m(x)\Big)\right].
 \end{align*}

 \end{enumerate}
\end{theorem}

\begin{proof}
First we show that our basic assumptions imply that $\overline{f_i}$ is proper for every $i\in\{1,2,\ldots,m\}$.
Let $i\in\{1,2,\ldots,m\}$.

Since $\varnothing\neq\big(\bigcap_{j=1}^{m}\dom f_j\big)\subseteq\big(\bigcap_{j=1}^{m}\dom\overline {f_j}\big)$, then
$\bigcap_{j=1}^{m}\dom\overline {f_j}\neq\varnothing$.
Let $x_0\in \bigcap_{i=j}^{m}\dom\overline {f_j}$. Suppose to the contrary that
$\overline{f_i}$ is not proper and thus there exists $y_0\in X$ such that
$\overline{f_i}(y_0)=-\infty$. Then by \cite[Proposition~2.4]{EkeTem}, $\overline{f_i}(x_0)=-\infty$. By the assumption,
$\overline{f_i}(x_0)=f_i(x_0)>-\infty$, which is a contradiction. Hence $\overline{f_i}$ is proper.

\ref{mainGed:1}$\Rightarrow$\ref{mainGed:2}:
Let $x^*\in X^*$.
Clearly, we have $\left(f_1^*\Box\cdots \Box f_m^*\right)(x^*)\geq \left(\sum_{i=1}^m f_i\right)^*(x^*)$. It suffices to show that
\begin{align}\left(\sum_{i=1}^m f_i\right)^*(x^*)\geq\left(f_1^*\Box\cdots \Box f_m^*\right)(x^*).\label{Redua:L1}\end{align}
First we show that
\begin{align}
\sum_{i=1}^m\overline{f_i}(y)\geq \sum_{i=1}^m f_i(y),\quad \forall y\in X.\label{Redua:L2}
\end{align}
Indeed, let $y\in X$. If $y\not\in\bigcap_{i=1}^{m}\dom\overline
{f_i}$.  Clearly, \eqref{Redua:L2} holds. Now assume that
$y\in\bigcap_{i=1}^{m}\dom\overline {f_i}$.  By our assumption
$\overline {f_i}(y)=f_i(y)$, we conclude that \eqref{Redua:L2}
holds. Combining both cases, we conclude that \eqref{Redua:L2} holds everywhere.

Since $\sum_{i=1}^m\overline{f_i}\leq \sum_{i=1}^m f_i$,  \eqref{Redua:L2} implies that
\begin{align}
\sum_{i=1}^m\overline{f_i}= \sum_{i=1}^m f_i. \label{Redua:L3d}
\end{align}
Taking the lower semicontinuous hull in the equality above, we have
\begin{align}
\sum_{i=1}^m\overline{f_i}= \sum_{i=1}^m f_i=\overline{f_1+\cdots+f_m}.\label{Redua:L3}
\end{align}

Clearly, if $(\sum_{i=1}^m f_i)^*(x^*)=+\infty$, then \eqref{Redua:L1}
holds.  Now assume that $(\sum_{i=1}^m f_i)^*(x^*)<+\infty$.  Then we
have $(\sum_{i=1}^m f_i)^*(x^*)\in\RR$ and thus $x^*\in\dom
(\sum_{i=1}^m f_i)^*$. Since $(\sum_{i=1}^m f_i)^*$ is lower
semicontinuous, given $\varepsilon>0$, there exists $x\in X$ such
that $x\in \partial _{\varepsilon} (\sum_{i=1}^m f_i)^* (x^*)$. Then
\begin{align*}(\sum_{i=1}^m f_i)^*(x^*)+\overline{(\sum_{i=1}^m f_i)}(x)&=
(\sum_{i=1}^m f_i)^*(x^*)+(\sum_{i=1}^m f_i)^{**}(x)
\leq \langle x, x^*\rangle+\varepsilon.
\end{align*}
By \eqref{Redua:L3}, we have
\begin{align*}(\sum_{i=1}^m f_i)^*(x^*)+(\sum_{i=1}^m f_i)(x)=(\sum_{i=1}^m \overline{f_i})^*(x^*)+(\sum_{i=1}^m \overline{f_i})(x)\leq \langle x, x^*\rangle+\varepsilon.
\end{align*}
Hence \begin{align}x^*\in \partial_{\varepsilon} (\sum_{i=1}^m f_i)(x)\quad\text{and}\quad x^*\in \partial_{\varepsilon} (\sum_{i=1}^m \overline{f_i})(x).\label{Redua:L2a}
\end{align}
\allowdisplaybreaks
Next, we claim that there exists $K>0$ such that
\begin{align}
x^*\in \sum_{i=1}^m\partial_{Km\varepsilon}f_i(x).\label{GeDu:L1}
\end{align}
Set $f:=\overline{f_1}$, $g:=(\sum_{i=2}^m \overline{f_i})$, and $\eta=\varepsilon$ in
Fact~\ref{BotCsCo}, and use \eqref{Redua:L2a} to write
\begin{align*}
&x^*\in \partial_{\varepsilon}(\sum_{i=1}^m \overline{f_i})(x)
\Rightarrow x^*\in \overline{ \left[
\partial_{2\varepsilon} \overline{f_1}(x)+\partial_{2\varepsilon} (\sum_{i=2}^m \overline{f_i})(x)\right]
 }^{\wk}.
\end{align*}
We repeat the same idea with $f:=\overline{f_2}$, $g:=(\sum_{i=3}^m \overline{f_i})$ in
Fact~\ref{BotCsCo}, and continue iteratively to obtain
\begin{align*}
 &\Rightarrow x^*\in \overline{ \left[
    \partial_{2\varepsilon} \overline{f_1}(x)+\overline{
      \partial_{3\varepsilon}
      \overline{f_2}(x)+\partial_{3\varepsilon}(\sum_{i=3}^m
      \overline{f_i})(x)}^{\wk}\right] }^{\wk}\\ &\Rightarrow x^*\in
\overline{ \left[ \partial_{2\varepsilon} \overline{f_1}(x)+
    \partial_{3\varepsilon}
    \overline{f_2}(x)+\partial_{3\varepsilon}(\sum_{i=3}^m
    \overline{f_i})(x)\right] }^{\wk}\\ &\quad \cdots\\ &\Rightarrow
x^*\in \overline{ \left[ \partial_{2\varepsilon} \overline{f_1}(x)+
    \partial_{3\varepsilon}
    \overline{f_2}(x)+\cdots+\partial_{m\varepsilon}\overline{f_m}(x)\right]
}^{\wk}\\  &\Rightarrow x^*\in \overline{ \left[
    \partial_{2\varepsilon} f_1(x)+ \partial_{3\varepsilon}
    f_2(x)+\cdots+\partial_{m\varepsilon}f_m(x)\right] }^{\wk}\quad
\text{(by \eqref{Redua:L2a} and $\overline{f_i}(x)=f_i(x),\forall
  i$)}\\ &\Rightarrow x^*\in \overline{ \left[ \partial_{m\varepsilon}
    f_1(x)+ \partial_{m\varepsilon}
    f_2(x)+\cdots+\partial_{m\varepsilon}f_m(x)\right]
}^{\wk}.
\end{align*}
By assumption \ref{mainGed:1}, the last inclusion implies that there exists $K>0$ such that
\begin{align*}
&x^*\in \partial_{Km\varepsilon} f_1(x)+
\partial_{Km\varepsilon}
f_2(x)+\cdots+\partial_{Km\varepsilon}f_m(x)\quad
 \end{align*}
Hence \eqref{GeDu:L1} holds. Thus, there exists $y^*_i\in\partial_{K\,m\varepsilon} f_i(x)$ such that
$x^*=\sum_{i=1}^{m} y^*_i$ and
\begin{align*}
f_i^*(y_i^*)+f_i(x)&\leq\langle x, y^*_i\rangle +Km\varepsilon,\quad
\forall i\in\{1,2,\ldots,m\}.
\end{align*}
Thus,
\begin{align*}(f_1^*\Box\cdots \Box f_m^*)(x^*)\leq
\sum_{i=1}^m f_i^*(y_i^*)&\leq-\sum_{i=1}^m f_i(x)+\langle x, x^*\rangle +Km^2\varepsilon\\
&\leq
(\sum_{i=1}^m f_i)^* (x^*) +Km^2\varepsilon.
\end{align*}
Letting $\varepsilon\longrightarrow 0$ in the above inequality, we have
\begin{align*}\left(f_1^*\Box\cdots \Box f_m^*\right)(x^*)\leq
\left(\sum_{i=1}^m f_i\right)^* (x^*).
\end{align*}
Hence \eqref{Redua:L1} holds and so
\begin{align}\left(\sum_{i=1}^m f_i\right)^*=f_1^*\Box\cdots \Box f_m^*.\end{align}

\ref{mainGed:2}$\Rightarrow$\ref{mainGed:3}: This clearly follows from the
lower semicontinuity of $(\sum_{i=1}^mf_i)^*$.

\ref{mainGed:3}$\Rightarrow$\ref{mainGed:1}:
Let $x\in \bigcap_{i=1}^{m}\dom f_i$ and $\varepsilon>0$, and
$x^*\in\overline{\left[\sum_{i=1}^{m}\partial_\varepsilon f_i(x)\right]}^{\wk}$.
Then for each $i=1,\ldots,m$ there exists a net $(x^*_{i,\alpha})_{\alpha\in I}$ in $\partial_\varepsilon f_i(x)$ such that
\begin{align}
\sum_{i=1}^{m}x^*_{i,\alpha}\weakstarly x^*.\label{GeDu:DL5}
\end{align}
We have
\begin{align}
f_i(x)+ f^*_i(x^*_{i,\alpha})\leq\langle x, x^*_{i,\alpha}\rangle+\varepsilon,\quad
\forall i\in\{1,2,\ldots,m\}\quad\forall \alpha\in I.
\end{align}
Thus
\begin{align}
\sum_{i=1}^mf_i(x)+(f_1^*\Box\cdots \Box f_m^*) (\sum_{i=1}^m x^*_{i,\alpha})\leq
\sum_{i=1}^mf_i(x)+ \sum_{i=1}^m f^*_i(x^*_{i,\alpha})\leq\langle x, \sum_{i=1}^m x^*_{i,\alpha}\rangle+m\varepsilon,\quad
\forall \alpha\in I.\label{GeDu:DL6}
\end{align}
Since $f_1^*\Box\cdots \Box f_m^*$ is
weak$^*$ lower semicontinuous,  it follows from \eqref{GeDu:DL6} and
\eqref{GeDu:DL5} that
\begin{align}
\sum_{i=1}^mf_i(x)+(f_1^*\Box\cdots \Box f_m^*) (x^*)
\leq\langle x, x^*\rangle+m\varepsilon.\label{GeDu:DL7}
\end{align}
There exists $y^*_i\in X^*$ such that
$\sum_{i=1}^m y^*_i= x^*$ and $\sum_{i=1}^m f^*_i(y^*_i)\leq
(f_1^*\Box\cdots \Box f_m^*) (x^*)+\varepsilon$.
Then by \eqref{GeDu:DL7},
\begin{align*}
\sum_{i=1}^mf_i(x)+\sum_{i=1}^m f^*_i(y^*_i)
\leq\langle x, x^*\rangle+(m+1)\varepsilon.
\end{align*}
Thus, we have
\begin{align*}
y^*_i\in\partial_{(m+1)\varepsilon} f_i(x),\quad
\forall i\in\{1,2,\ldots,m\}.
\end{align*}
Hence
\begin{align*}
x^*=\sum_{i=1}^m y^*_i\in\sum_{i=1}^m\partial_{(m+1)\varepsilon} f_i(x),
\end{align*}
and the statement in (i) holds for $K:=(m+1)$.

\ref{mainGed:2}$\Rightarrow$\ref{mainGed:4}:
Let $x\in X$ and $\varepsilon\geq0$.

We have \begin{align*}
&\bigcap_{\eta>0} \left[ \bigcup_{\varepsilon_i\geq0, \, \sum_{i=1}^m\varepsilon_i=\varepsilon+\eta}
 \Big(\partial_{\varepsilon_1} f_1(x)+\cdots+\partial_{\varepsilon_m} f_m(x)\Big)\right]\\
 &\subseteq \bigcap_{\eta>0} \left[ \bigcup_{\varepsilon_i\geq0, \, \sum_{i=1}^m\varepsilon_i=\varepsilon+\eta}
 \partial_{\sum_{i}^m\varepsilon_i}(f_1+\cdots+f_m)(x)\right]= \bigcap_{\eta>0}\partial_{\varepsilon+\eta}(f_1+\cdots+f_m)(x)\\
&=\partial_{\varepsilon}(f_1+\cdots+f_m)(x).
\end{align*} Now we show the other inclusion:
\begin{align}
\partial_{\varepsilon}(f_1+\cdots+f_m)(x)\subseteq\Big(\bigcap_{\eta>0} \left[ \bigcup_{\varepsilon_i\geq0, \, \sum_{i=1}^m\varepsilon_i=\varepsilon+\eta}
 \Big(\partial_{\varepsilon_1} f_1(x)+\cdots+\partial_{\varepsilon_m} f_m(x)\Big)\right]\Big).\label{MainEx:1}
\end{align}
Let $x^*\in \partial_{\varepsilon}(f_1+\cdots+f_m)(x)$. Then  we have
$\sum_{i=1}^m f_i(x)+(\sum f_i)^*(x^*)\leq\langle x, x^*\rangle+\varepsilon$.
By \ref{mainGed:2}, we have
\begin{align}
\sum_{i=1}^m f_i(x)+\big(f_1^*\Box\cdots \Box f_m^*\big)(x^*)\leq\langle x, x^*\rangle+\varepsilon.
\label{MainEx:2}
\end{align}
Let $\eta>0$. Then there exists $y^*_i\in X^*$ such that $\sum_{i=1}^m y^*_i=x^*$
and
$\sum_{i=1}^m f^*_i(y^*_i)\leq \big(f_1^*\Box\cdots \Box f_m^*\Big)(x^*)+\eta$. Then by \eqref{MainEx:2},
\begin{align}
\sum_{i=1}^m f_i(x)+\sum_{i=1}^m f^*_i(y^*_i)\leq\langle x, x^*\rangle+\varepsilon+\eta.
\label{MainEx:3}
\end{align}
Set $\gamma_i:=f_i(x)+f^*_i(y^*_i)-\langle x, y^*_i\rangle$.
Then $\gamma_i\geq0$ and $y^*_i\in\partial_{\gamma_i} f_i(x)$.  By \eqref{MainEx:3},
\begin{align}
\langle x, x^*\rangle+\sum_{i=1}^m \gamma_i=\sum_{i=1}^m \left[\langle x, y^*_i\rangle+\gamma_i\right]
\leq\langle x, x^*\rangle+\varepsilon+\eta.
\end{align}
Hence $\sum_{i=1}^m \gamma_i\leq\varepsilon+\eta$.
Set $\varepsilon_1:=\varepsilon+\eta-\sum_{i=2}^m \gamma_i$ and $\varepsilon_i:=\gamma_i$ for every $i=\{2,3,\ldots,m\}$.
Then $\varepsilon_1\geq\gamma_1$ and we have
\[
x^*=\sum_{i=1}^m y^*_i\in  \sum_{i=1}^m \partial_{\varepsilon_i} f_i(x).
\]
Hence $x^*\in \bigcup_{\varepsilon_i\geq0, \, \sum_{i=1}^m\varepsilon_i=\varepsilon+\eta}
 \Big(\partial_{\varepsilon_i} f_i(x)+\cdots+\partial_{\varepsilon_m} f_m(x)\Big)$ and therefore \eqref{MainEx:1} holds.

\ref{mainGed:4}$\Rightarrow$\ref{mainGed:1}: Let $x\in \bigcap_{i=1}^{m}\dom f_i$, $\varepsilon>0$, and
$x^*\in\overline{\left[\sum_{i=1}^{m}\partial_\varepsilon f_i(x)\right]}^{\wk}$.
Then for each $i=1,\ldots,m$ there exists a net $(x^*_{i,\alpha})_{\alpha\in I}$ in $\partial_\varepsilon f_i(x)$ such that
\begin{align}
\sum_{i=1}^{m}x^*_{i,\alpha}\weakstarly x^*,\label{MainEx:5}
\end{align}
and this implies that
\begin{align}
\sum_{i=1}^{m}x^*_{i,\alpha}\in\Big(\bigcap_{\eta>0} \left[ \bigcup_{\varepsilon_i\geq0, \, \sum_{i=1}^m\varepsilon_i=m\varepsilon+\eta}
 \Big(\partial_{\varepsilon_1} f_1(x)+\cdots+\partial_{\varepsilon_m} f_m(x)\Big)\right]\Big).
\end{align}
Assumption \ref{mainGed:4} yields $\sum_{i=1}^{m}x^*_{i,\alpha}\in\partial_{m\varepsilon}(f_1+\cdots+f_m)(x)$.
Since $\partial_{m\varepsilon}(f_1+\cdots+f_m)(x)$ is weak$^*$ closed,
\eqref{MainEx:5} shows that $x^*\in \partial_{m\varepsilon}(f_1+\cdots+f_m)(x)$. Using \ref{mainGed:4} again for $\eta=\varepsilon$, we conclude that
$x^*\in\Big(\partial_{(m+1)\varepsilon} f_i(x)+\cdots+\partial_{(m+1)\varepsilon} f_m(x)\Big)$.

Therefore, statement \ref{mainGed:1} holds for $K:=m+1$.
\end{proof}

\begin{remark} \begin{enumerate}
 \item [(a)]
We point out that the proof of Theorem~\ref{Theman:1}\ref{mainGed:1}
actually shows that $K=m+1$, and this constant is {\em independent of the
functions} $f_1,\ldots,f_m$.
\item[(b)]
Part \ref{mainGed:1} implies \ref{mainGed:2} of Theorem~\ref{Theman:1} generalizes \cite[Proposition~4.1]{BerT},
\cite[Theorem~3.2]{BotCse2} by Bo\c{t} and Csetnek, and \cite[Theorem~3.6]{BurMaj} by
Burachik and Majeed.
\item[(c)]
  A result similar to
  Theorem~\ref{Theman:1}\ref{mainGed:3}$\Leftrightarrow$\ref{mainGed:4}
  has been established in \cite[Corollary~3.9]{BotGrad} by Bo\c{t} and
  Grad.
  
  \end{enumerate}
\end{remark}
An immediate corollary follows:

\begin{corollary}\label{TheodualM:1}
Let $f, g:X\rightarrow\RX$ be proper convex with $\dom f\cap\dom g\neq\varnothing$. Suppose that $\overline {f}=f$ and  $\overline {g}=g$ on $\dom\overline {f}\cap\dom\overline {g}$. Suppose also that
for every $x\in \dom f\cap\dom g$ and $\varepsilon>0$,
\begin{align*}
\partial_\varepsilon f(x)+\partial_\varepsilon g(x)\quad\text{is weak$^*$ closed}.
\end{align*}

Then $(f+g)^*=f^*\Box g^*$ in $X^*$.
Consequently, $\inf (f+g)=\sup_{x^*\in X^*}\{-f^*(x^*)-g^*(-x^*)\}$.
\end{corollary}

Note that, for a linear subspace $S\subseteq X$, we have
$\partial_\varepsilon \iota_S =S^{\perp}$. Taking this into account we
derive the Bertsekas Constraint Qualification result from Theorem ~\ref{Theman:1}.

%% \begin{corollary}
%% Let $m\in\NN$, and $f_i:X\rightarrow\RX$ be proper convex, where
%% $i\in\{1,2,\ldots,m\}$. Let $S$ be a closed linear subspace of $X$ and
%% $S^{\perp}\subseteq X^*$ its orthogonal with $S\cap \big(\bigcap_{i=1}^{m}\dom f_i\big)\neq\varnothing$.  Suppose
%%  that $\overline
%% {f_i}=f_i$ on $S\cap \big(\bigcap_{i=1}^{m}\dom f_i\big)$.  Assume that for every
%% $x\in S\cap \big(\bigcap_{i=1}^{m}\dom{f_i}\big)$ and for every $\varepsilon>0$ we have that
%% \begin{align*}
%% S^{\perp}+\sum^m_{i=1}\partial_\varepsilon f_i(x)\quad\text{is weak$^*$ closed}.
%% \end{align*}
%% Then $(\iota_S+\sum^m_{i=1} f_i)^*= (\iota_S)^*\Box f_1^*\Box\cdots\Box f_m^*$ in $X^*$.
%% Consequently, $v(P)=\inf_{x\in S} \{\sum^m_{i=1} f_i(x)\}=\sup_{x^*\in S^{\perp}}\{-\sum^m_{i=1}f_i^*(x^*)\}=v(D)$.
%% \end{corollary}

\begin{corollary}[Bertsekas]\emph{(See \cite[Proposition~4.1]{BerT}.)}\label{TheodualM:2}
Let $m\in\NN$ and suppose that $X_i$ is a finite dimensional space,  and let $f_i:X_i\rightarrow\RX$ be proper lower semicontinuous and  convex, where
$i\in\{1,2,\ldots,m\}$. Let $S$ be a linear subspace of $X_1\times X_2\times\cdots\times X_m$ with
 $S\cap \big(\bigcap_{i=1}^{m}\dom f_i\big)\neq\varnothing$.   Define $g_i:X_1\times X_2\times \cdots \times X_m\rightarrow\RX$ by  $g_i(x_1,\ldots, x_m):=f_i(x_i)$. Assume that for every
$x\in S\cap \big(\bigcap_{i=1}^{m}\dom{f_i}\big)$ and for every $\varepsilon>0$ we have that
\begin{align*}
S^{\perp}+\sum^m_{i=1}\partial_\varepsilon g_i(x)\quad\text{is  closed}.
\end{align*}
Then  $v(P)=\inf_{x\in S} \{\sum^m_{i=1} f_i(x)\}=\sup_{x^*\in S^{\perp}}\{-\sum^m_{i=1}f_i^*(x^*)\}=v(D)$.
\end{corollary}

The following example which is due to \cite[Example~3.1]{BurMaj} and \cite[Example, page~2798]{BotWank1},
 shows that the infimal convolution  in Corollary~\ref{TheodualM:1} is not always achieved (exact).

\begin{example}\label{ExamDul:1}
Let $X=\RR^2$, and  $f:=\iota_C, g:=\iota_D$, where
$C:=\{(x,y)\in\RR^2\mid 2x+y^2\leq0\}$ and $D:=\{(x,y)\in\RR^2\mid x\geq0\}$.  Then $f$ and $g$ are proper lower semicontinuous and convex with $\dom f\cap\dom g=\{(0,0)\}$. For every $\varepsilon>0$, $\partial_\varepsilon f(0,0)+\partial_\varepsilon g(0,0)$ is closed.
Hence $(f+g)^*=f^*\Box g^*$. But
$f^*\Box g^*$ is not exact everywhere and $\partial (f+g)(0)\neq\partial f(0)+\partial g(0)$. Consequently, $\epi f^*+\epi g^*$ is not closed in the topology $\omega (X^*, X)\times \RR$.

\end{example}
\begin{proof}
Clearly, $f$ and $g$ are proper lower semicontinuous convex.
Let $\varepsilon>0$. Then
by  \cite[Example~3.1]{BurMaj}
\begin{align}\label{Redua:9}
\partial_\varepsilon f(0,0)=\bigcup_{u\geq0} \big(u\times\left[-\sqrt{2\varepsilon u},\sqrt{2\varepsilon u}\right]\big)
\quad\text{and}\quad
\partial_\varepsilon g(0,0)=\left]-\infty,0\right]\times\{0\}.
\end{align}
Thus, $\partial_\varepsilon f(0,0)+\partial_\varepsilon g(0,0)=\RR^2$ and then $\partial_\varepsilon f(0,0)+\partial_\varepsilon g(0,0)$ is closed.  Corollary~\ref{TheodualM:1} implies that $(f+g)^*=f^*\Box g^*$.
\cite[Example, page~2798]{BotWank1} shows that
$(f^*\Box g^*)$ is not exact at $(1,1)$ and $\partial (f+g)(0)\neq\partial f(0)+\partial g(0)$.  By \cite{BurJeya, BotWank1}, $\epi f^*+\epi g^*$ is not closed in the topology $\omega (X^*, X)\times \RR$.
\end{proof}

The following result is classical, we state and prove it here for more
convenient and clear future use.

\begin{lemma}[Hiriart-Urruty]\label{Lemresou}
Let $m\in\NN$, and
 $f_i:X\rightarrow\RX$ be proper convex with $\bigcap_{i=1}^{m}\dom f_i\neq\varnothing$,
 where $i\in\{1,2,\ldots,m\}$.
Assume that $(\sum_{i=1}^m f_i)^*=f_1^*\Box\cdots \Box f_m^*$  in $X^*$ and the infimal convolution is exact (attained) everywhere.
Then \[\partial (f_1+f_2+\cdots+f_m)=\partial f_1+\cdots+\partial f_m.\]
\end{lemma}
\begin{proof}
Let $x\in X$. We always have $\partial (f_1+f_2+\cdots+f_m)(x)\supseteq\partial f_1(x)+\cdots+\partial f_m(x)$.
So it suffices to show that
\begin{align}
\partial (f_1+f_2+\cdots+f_m)(x)\subseteq\partial f_1(x)+\cdots+\partial f_m(x).\label{ExaEpG:3}
\end{align}
Let $w^*\in\partial (f_1+f_2+\cdots+f_m)(x)$. Then
\begin{align*}
(f_1+f_2+\cdots+f_m)(x)+(f_1+f_2+\cdots+f_m)^*(w^*)=
\langle x, w^*\rangle.
\end{align*}
By the assumption, there exists $w^*_i\in X^*$ such that $\sum_{i=1}^m w^*_i=w^*$ and
\begin{align*}
f_1(x)+f_2(x)+\cdots+f_m(x)+f_1^*(w_1^*)+\cdots+f_m^*(w_m^*)=
\langle x, w_1^*+\cdots+w_m^*\rangle.
\end{align*}
Hence
\begin{align*}
w_i^*\in\partial f_i(w_i),\quad\forall i\in\{1,2,\ldots,m\}.
\end{align*}
Thus
\begin{align*}
w^*=\sum_{i=1}^m w^*_i\in\sum_{i=1}^m\partial f_i(w_i),
\end{align*}
and \eqref{ExaEpG:3} holds.
\end{proof}

A less immediate corollary is:

\begin{corollary}\label{DualCorL:1}\emph{(See \cite[Theorem~3.5.8]{BotGradWank}.)}
Let $m\in\NN$, and
 $f_i:X\rightarrow\RX$ be proper convex with $\bigcap_{i=1}^{m}\dom f_i\neq\varnothing$,
 where $i\in\{1,2,\ldots,m\}$. Suppose  that $\overline {f_i}=f_i$  on $\bigcap_{i=1}^{m}\dom\overline {f_i}$.
 Assume that $\epi f^*_1+\cdots+\epi f^*_m$ is closed in
  the topology $\omega (X^*, X)\times \RR$.

   Then
$(\sum_{i=1}^m f_i)^*=f_1^*\Box\cdots \Box f_m^*$  in $X^*$ and the infimal convolution is exact (attained) everywhere.
In consequence, we also have \[\partial (f_1+f_2+\cdots+f_m)=\partial f_1+\cdots+\partial f_m.\]
\end{corollary}

\begin{proof}
Let $x\in \bigcap_{i=1}^{m}\dom f_i$, $x^*\in \overline{\left[\sum_{i=1}^{m}\partial_\varepsilon f_i(x)\right]}^{\wk}$ and $\varepsilon>0$.  We will show that
\begin{align}
x^*\in\sum_{i=1}^{m}\partial_{m\varepsilon} f_i(x).\label{GeDu:L10a}
\end{align}
The assumption on $x^*$ implies that for each $i=1,\ldots,m$ there exists $(x^*_{i,\alpha})_{\alpha\in I }$ in $\partial_\varepsilon f_i(x)$ such that
\begin{align}
\sum_{i=1}^{m}x^*_{i,\alpha}\weakstarly x^*.\label{GeDu:L10}
\end{align}
We have
\begin{align}
 f^*_i(x^*_{i,\alpha})\leq-f_i(x)+\langle x, x^*_{i,\alpha}\rangle+\varepsilon,\quad
\forall i\in\{1,2,\ldots,m\}\quad\forall \alpha\in I.
\end{align}
Thus $(x^*_{i,\alpha}, -f_i(x)+\langle x, x^*_{i,\alpha}\rangle+\varepsilon)\in\epi f_i^*, \forall i$
and hence
\begin{align} \Big(\sum_{i=1}^m x^*_{i,\alpha}, -\sum_{i=1}^m f_i(x)+\langle x, \sum_{i=1}^m x^*_{i,\alpha}\rangle+m\varepsilon\Big)\in\epi f^*_1+\cdots+\epi f^*_m.\label{GeDu:L11}
\end{align}
Now $\epi f^*_1+\cdots+\epi f^*_m$ is closed in the topology $\omega (X^*, X)\times \RR$.
Thus, by \eqref{GeDu:L10} and \eqref{GeDu:L11},
we have
\begin{align}
 \Big(x^*, -\sum_{i=1}^m f_i(x)+\langle x, x^*\rangle+m\varepsilon\Big)\in\epi f^*_1+\cdots+\epi f^*_m.
\end{align}
Consequently, there exists $y^*_i\in X^*$ and $t_i\geq0$ such that
\begin{align}
x^*&=\sum_{i=1}^m y^*_i\label{GeDu:L12}\\
 -\sum_{i=1}^m f_i(x)+\langle x, x^*\rangle+m\varepsilon&=
\sum_{i=1}^m (f^*(y_i^*)+t_i).\nonumber
\end{align}
Hence
\begin{align}
 -\sum_{i=1}^m f_i(x)+\langle x, x^*\rangle+m\varepsilon\geq
\sum_{i=1}^m f^*(y_i^*).
\end{align}
Then we have
\begin{align*}
y_i^*\in\partial_{m\varepsilon} f_i(x),\quad
\forall i\in\{1,2,\ldots,m\}.
\end{align*}
Thus by \eqref{GeDu:L12},
\begin{align*}
x^*\in\sum_{i=1}^m\partial_{m\varepsilon} f_i(x).
\end{align*}
Hence \eqref{GeDu:L10a} holds. Applying Theorem~\ref{Theman:1}, part \ref{mainGed:1} implies \ref{mainGed:2}, we have
\begin{align}(\sum_{i=1}^m f_i)^*=f_1^*\Box\cdots \Box f_m^*.
\label{ExaEpG:1}
\end{align}
Let $z^*\in X^*$. Next we will show that
$(f_1^*\Box\cdots \Box f_m^*)(z^*)$ is achieved.
If $z^*\notin\dom (\sum_{i=1}^m f_i)^*$, then
$(f_1^*\Box\cdots \Box f_m^*)(x^*)=+\infty$ by \eqref{ExaEpG:1}
and hence $(f_1^*\Box\cdots \Box f_m^*)(z^*)$ is achieved.

Now suppose that $z^*\in\dom (\sum_{i=1}^m f_i)^*$  and then
$(\sum_{i=1}^m f_i)^*(z^*)\in\RR$.
By \eqref{ExaEpG:1}, there exists $(z^*_{i,n})_{n\in\NN}$ such that $\sum_{i=1}^m z^*_{i,n}= z^*$
and
\begin{align*}(\sum_{i=1}^m f_i)^*(z^*)\leq
f^*_1(z^*_{1,n})+f^*_2(z^*_{2,n})+\cdots+f^*_m(z^*_{m,n})
\leq(\sum_{i=1}^m f_i)^*(z^*)+\frac{1}{n}.
\end{align*}
Then we have
\begin{align}f^*_1(z^*_{1,n})+f^*_2(z^*_{2,n})+\cdots+f^*_m(z^*_{m,n})
\longrightarrow
(\sum_{i=1}^m f_i)^*(z^*).\label{ExaEpG:2}
\end{align}
Since $\big(z^*, \sum_{i=1}^m f^*_i(z^*_{i,n})\big)=
\big(\sum_{i=1}^m z^*_{i,n}, \sum_{i=1}^m f^*_i(z^*_{i,n})\big)\in
\epi f^*_1+\cdots+\epi f^*_m$ and $\epi f^*_1+\cdots+\epi f^*_m$
is closed
 in
  the topology $\omega (X^*, X)\times \RR$, \eqref{ExaEpG:2} implies that
\begin{align*}\big(z^*, (\sum_{i=1}^m f_i)^*(z^*)\big)\in
\epi f^*_1+\cdots+\epi f^*_m.
\end{align*}
Thus, there exists $v^*_i\in X^*$ such that $\sum_{i=1}^m v^*_i=z^*$ and
\begin{align}
 (\sum_{i=1}^m f_i)^*(z^*)\geq \sum_{i=1}^m f^*_i(v^*_i)\geq (f_1^*\Box\cdots \Box f_m^*)(z^*).
 \label{ExaEpG:3a}
\end{align}
Since $ (\sum_{i=1}^m f_i)^*(z^*)=(f_1^*\Box\cdots \Box f_m^*)(z^*)$ by
\eqref{ExaEpG:1}, it follows from \eqref{ExaEpG:3a} that
$ (\sum_{i=1}^m f_i)^*(z^*)= \sum_{i=1}^m f^*_i(v^*_i)$.
Hence $(f_1^*\Box\cdots \Box f_m^*)(z^*)$ is achieved.

The applying Lemma~\ref{Lemresou}, we have
$\partial (f_1+f_2+\cdots+f_m)=\partial f_1+\cdots+\partial f_m$.
\end{proof}

When there are precisely two functions this reduces to:

\begin{corollary}[Bo\c{t} and Wanka]\emph{(See \cite[Theorem~3.2]{BotWank1}.)}
\label{DualCorL:2}
Let
 $f,g:X\rightarrow\RX$ be proper lower semicontinuous and convex with $\dom f\cap\dom g\neq\varnothing$.
 Assume that $\epi f^*+\epi g^*$ is closed in
  the topology $\omega (X^*, X)\times \RR$. Then
$(f+g)^*=f^*\Box g^*$ in $X^*$ and the infimal convolution is exact everywhere.
In consequence, $\partial (f+g)=\partial f+\partial g$.
\end{corollary}
\begin{proof}
Directly apply Corollary~\ref{DualCorL:1}.
\end{proof}
\begin{remark}
In the setting of Banach space,  Corollary~\ref{DualCorL:2}
was first established by Burachik and Jeyakumar \cite{BurJeya}.
Example~\ref{ExamDul:1} shows that the equality $(f+g)^*=f^*\Box g^*$ is not a sufficient condition for $\epi f^*+\epi
g^*$ to be closed.\end{remark}

The following result, stating the equivalence between the closed epigraph condition
and condition (ii) in Theorem \ref{Theman:1} {\em with
exactness}, is well known but hard to track down.

\begin{proposition}\label{35}
  Let $m\in\NN$, and $f_i:X\rightarrow\RX$ be proper lower
  semicontinuous and convex with $\bigcap_{i=1}^{m}\dom
  f_i\neq\varnothing$, where $i\in\{1,2,\ldots,m\}$.  Then $\epi
  f^*_1+\cdots+\epi f^*_m$ is closed in the topology $\omega (X^*,
  X)\times \RR$ if and only if $(\sum_{i=1}^m f_i)^*=f_1^*\Box\cdots
  \Box f_m^*$ in $X^*$ and the infimal convolution is exact.
\end{proposition}
\begin{proof}
$\Rightarrow$:  This follows directly from Corollary~3.10.

$\Leftarrow$: Assume now that $(\sum_{i=1}^m f_i)^*=f_1^*\Box\cdots
\Box f_m^*$ in $X^*$ and the infimal convolution is always exact. Note
that this assumption implies that the function $f_1^*\Box\cdots \Box
f_m^*$ is lower semicontinuous in $X^*$. Let $(w^*, r)\in
X^*\times\RR$ be in the closure of $\epi f^*_1+\cdots+\epi f^*_m$ in
the topology $\omega (X^*, X)\times \RR$.  We will show that $(w^*,
r)\in\epi f^*_1+\cdots+\epi f^*_m$. The assumption on $(w^*, r)$
implies that there exist $(x^*_{i,\alpha})_{\alpha\in I }$ in $\dom
f^*_i$ and $(r_{i,\alpha})_{\alpha\in I }$ in $\RR$ such that
\begin{align}
  w^*_{\alpha}:=\sum_{i=1}^{m}x^*_{i,\alpha}\weakstarly w^*,\quad
 f^*_i(x^*_{i,\alpha}) \le r_{i,\alpha}\,,\,\forall \, i,\,\alpha \quad\text{and} \quad
  \sum_{i=1}^{m}r_{i,\alpha} \longrightarrow
  r.\label{Epig:ade1}
\end{align}
Then
\begin{align}
\Big(f_1^*\Box\cdots \Box f_m^*\Big)(w^*_{\alpha})\leq \sum_{i=1}^{m}
f^*_i(x^*_{i,\alpha})\le \sum_{i=1}^{m}r_{i,\alpha}.\label{Epig:ade2}
\end{align}
Our assumption implies that $f_1^*\Box\cdots \Box f_m^*$ is
lower semicontinuous, hence by taking limits in \eqref{Epig:ade2} and
using \eqref{Epig:ade1} we obtain
\begin{align}
\Big(f_1^*\Box\cdots \Box f_m^*\Big)(w^*) \leq r.\label{Epig:ade3}
\end{align}
By assumption, $\Big(f_1^*\Box\cdots \Box f_m^*\Big)(w^*)$ is
exact. Therefore there exists $w^*_i$ such that
$w^*=\sum_{i=1}^mw^*_i$ and $\Big(f_1^*\Box\cdots \Box
f_m^*\Big)(w^*)=\sum_{i=1}^m f_i^*(w^*_i)$.  The latter fact
and \eqref{Epig:ade3} show that $(w^*,r)\in\epi f^*_1+\cdots+\epi
f^*_m$.
\end{proof}

We next dualize Corollary~\ref{DualCorL:1}.

\begin{corollary}[Dual conjugacy]\label{DualCorL:DU1} Suppose that $X$ is a reflexive Banach space.
Let $m\in\NN$, and
 $f_i:X\rightarrow\RX$ be proper lower semicontinuous and convex with $\bigcap_{i=1}^{m}\dom f^*_i\neq\varnothing$,
 where $i\in\{1,2,\ldots,m\}$.
 Assume that $\epi f_i+\cdots+\epi f_m$ is closed in
  the weak topology $\omega (X, X^*)\times \RR$.

   Then
$(\sum_{i=1}^m f^*_i)^*=f_1\Box\cdots \Box f_m$  in $X$ and the infimal convolution is exact (attained) everywhere.
In consequence, we also have \[\partial (f^*_1+f^*_2+\cdots+f^*_m)=\partial f^*_1+\cdots+\partial f^*_m.\]
\end{corollary}
\begin{proof}
Apply  Corollary~\ref{DualCorL:1} to the functions $f^*_i$.
\end{proof}

In a Banach space we can add a general interiority condition for closure.

\begin{remark}[Transversality]
Suppose that $X$ is a Banach space,
 and let $f, g$ be defined as in Corollary~\ref{DualCorL:2}.
  If $\bigcup_{\lambda>0}\lambda\left[\dom f-\dom g\right]$
is a closed subspace, then the Attouch-Brezis theorem implies that
$\epi f^*+\epi g^*$ is closed in the topology $\omega (X^*, X)\times \RR$ \cite{AtBrezis,RodSim,BotWank1,BurJeya}. This result works also in a locally convex Fr\'echet space \cite{Bor}.
\end{remark}

The following result shows that sublinearity  rules out the pathology of Example~\ref{ExamDul:1} in Theorem~\ref{Theman:1}\ref{mainGed:1}.

\begin{corollary}[Sublinear functions]\label{sublcSE:4}
Let $m\in\NN$, and
 $f_i:X\rightarrow\RX$ be proper sublinear,
 where $i\in\{1,2,\ldots,m\}$. Suppose  that $\overline {f_i}=f_i$  on $\bigcap_{i=1}^{m}\dom\overline {f_i}$.
 Then the following eight conditions are equivalent.
 \begin{enumerate}
 \item \label{mainGsu:a1} There exists $K>0$ such that for every $x\in \bigcap_{i=1}^{m}\dom f_i$, and every $\varepsilon>0$,
\begin{align*}
\overline{\left[\sum_{i=1}^{m}\partial_\varepsilon f_i(x)\right]}^{\wk}\subseteq\sum_{i=1}^{m}\partial_{K\varepsilon} f_i(x).
\end{align*}
 \item \label{mainGsu:1}
$\sum_{i=1}^m\partial f_i(0)$ is weak$^*$ closed.

\item \label{mainGsu:2}$\left(\sum_{i=1}^m f_i\right)^*=f_1^*\Box\cdots \Box f_m^*$
in $X^*$.

\item \label{mainGsu:3} $f_1^*\Box\cdots \Box f_m^*$ is weak$^*$ lower semicontinuous.

\item \label{mainGsu:4a} For every $x\in X$ and $
 \varepsilon\geq0$,
  \begin{align*}
 \partial_{\varepsilon}(f_1+\cdots+f_m)(x)= \bigcap_{\eta>0} \left[ \bigcup_{\varepsilon_i\geq0, \, \sum_{i=1}^m\varepsilon_i=\varepsilon+\eta}
 \Big(\partial_{\varepsilon_1} f_1(x)+\cdots+\partial_{\varepsilon_m} f_m(x)\Big)\right].
 \end{align*}

\item \label{mainGsu:4} $\epi f^*_1+\cdots+\epi f^*_m$ is closed in
  the topology $\omega (X^*, X)\times \RR$.

\item \label{mainGsu:5}
$(\sum_{i=1}^m f_i)^*=f_1^*\Box\cdots \Box f_m^*$  in $X^*$ and the infimal convolution is exact (attained) everywhere it is finite.

\item \label{mainGsu:6} \[\partial (f_1+f_2+\cdots+f_m)=\partial f_1+\cdots+\partial f_m.\]

 \end{enumerate}
\end{corollary}

\begin{proof}
We first show that \ref{mainGsu:a1}$\Leftrightarrow$\ref{mainGsu:1}$\Leftrightarrow$\ref{mainGsu:2}$\Leftrightarrow$\ref{mainGsu:3}$\Leftrightarrow$\ref{mainGsu:4a}.
By Theorem~\ref{Theman:1}, it suffices to show that \ref{mainGsu:a1}$\Leftrightarrow$\ref{mainGsu:1}.

\ref{mainGsu:a1}$\Rightarrow$\ref{mainGsu:1}:
Let $x^*\in\overline{\left[\sum_{i=1}^{m}\partial f_i(0)\right]}^{\wk}$.
Then $x^*\in\overline{\left[\sum_{i=1}^{m}\partial_1 f_i(0)\right]}^{\wk}$.
By \ref{mainGsu:a1}, there exists $K>0$ such that
$x^*\in\sum_{i=1}^{m}\partial_K f_i(0)$. \cite[Theorem~2.4.14(iii)]{Zalinescu} shows that
$x^*\in\sum_{i=1}^{m}\partial f_i(0)$. Hence
$\sum_{i=1}^m\partial f_i(0)$ is weak$^*$ closed.

\ref{mainGsu:1}$\Rightarrow$\ref{mainGsu:a1}:
Let $x\in \bigcap_{i=1}^{m}\dom f_i$ and $\varepsilon>0$, and
$x^*\in\overline{\left[\sum_{i=1}^{m}\partial_\varepsilon f_i(x)\right]}^{\wk}$.
Then there exists a net $(x^*_{i,\alpha})_{\alpha\in I}$ in $\partial_\varepsilon f_i(x)$ such that
\begin{align}
\sum_{i=1}^{m}x^*_{i,\alpha}\weakstarly x^*.\label{GeDu:sul5}
\end{align}
Then by \cite[Theorem~2.4.14(iii)]{Zalinescu}, we have
\begin{align}x^*_{i,\alpha}\in\partial f_i(0)\quad\text{and}\quad
f_i(x)\leq\langle x, x^*_{i,\alpha}\rangle+\varepsilon,\quad
\forall i\in\{1,2,\ldots,m\}\quad\forall \alpha\in I.
\end{align}
Hence
\begin{align}\sum_{i=1}^mx^*_{i,\alpha}\in\sum_{i=1}^m\partial f_i(0)\quad\text{and}\quad
\sum_{i=1}^mf_i(x)\leq\langle x, \sum_{i=1}^mx^*_{i,\alpha}\rangle+m\varepsilon,\quad
\forall \alpha\in I.\label{GeDu:sul6}
\end{align}
Thus, by \eqref{GeDu:sul5} and \eqref{GeDu:sul6},
\begin{align}x^*\in\overline{\left[\sum_{i=1}^m\partial f_i(0)\right]}^{\wk}\quad\text{and}\quad
\sum_{i=1}^mf_i(x)\leq\langle x, x^*\rangle+m\varepsilon.\label{GeDu:sul7}
\end{align}
Since $\sum_{i=1}^m\partial f_i(0)$ is weak$^*$ closed, by \eqref{GeDu:sul7}, $x^*\in
\sum_{i=1}^m\partial f_i(0)$. Then there exists $y^*_i\in \partial f_i(0)$
such that
\begin{align}
x^*=\sum_{i=1}^m y^*_i.\label{GeDu:sul8}
\end{align}
By \eqref{GeDu:sul7} and \cite[Theorem~2.4.14(i)]{Zalinescu}, we have
\begin{align*}
\sum_{i=1}^m\big(f_i(x)+f_i^*(y^*_i)\big)=
\sum_{i=1}^m\big(f_i(x)+\iota_{\partial f_i(0)}(y^*_i)\big)\leq\langle x, x^*\rangle+m\varepsilon
\end{align*}
Hence \begin{align*}
y^*_i\in\partial_{m\varepsilon} f_i(x),\quad\forall i\in\{1,2,\ldots,m\}.
\end{align*}
Then by \eqref{GeDu:sul8}, $x^*\in\sum_{i=1}^m \partial_{m\varepsilon} f_i(x)$.
Setting $K:=m$, we obtain \ref{mainGsu:a1}.

Hence \ref{mainGsu:a1}$\Leftrightarrow$\ref{mainGsu:1}$\Leftrightarrow$\ref{mainGsu:2}
$\Leftrightarrow$\ref{mainGsu:3}$\Leftrightarrow$\ref{mainGsu:4a}.

\ref{mainGsu:1}$\Leftrightarrow$\ref{mainGsu:4}:
 By \cite[Theorem~2.4.14(i)]{Zalinescu},
 we have
 \begin{align*}
 \epi f^*_1+\cdots+\epi f^*_m=\Big(\partial f_1(0)+\cdots+\partial f_m(0)\Big)\times \{r\mid r\geq0\}.
 \end{align*}
The rest is now clear.

\ref{mainGsu:4}$\Rightarrow$\ref{mainGsu:5}: Apply  Corollary~\ref{DualCorL:1}.

\ref{mainGsu:5}$\Rightarrow$\ref{mainGsu:6}:
Apply Lemma~\ref{Lemresou} directly.

\ref{mainGsu:6}$\Rightarrow$\ref{mainGsu:1}: Since $\sum_{i=1}^m\partial f_i(0)
=\partial (f_1+f_2+\cdots+f_m)(0)$, we conclude that $\sum_{i=1}^m\partial f_i(0)$ is weak$^*$ closed
\end{proof}

\begin{remark}\rm
  By applying Corollary \ref{sublcSE:4} to a single sublinear
  function, we conclude that $f=\overline{f}$ and is lower
  semicontinuous everywhere (see \eqref{Redua:L3d}). By \cite[Theorem
  2.4.14]{Zalinescu}, this implies existence of subdifferentials at
  $0$ (as indeed can also be deduced from Corollary \ref{sublcSE:4}).
\end{remark}

\begin{corollary}[Burachik, Jeyakumar and Wu]\emph{(See \cite[Corollary~3.3]{BurJeyaWu}.)}
Suppose that $X$ is a Banach space. Let $f,g:X\rightarrow\RX$ be proper lower semicontinuous and sublinear.
Then the following are equivalent.
\begin{enumerate}

\item  $\epi f^*+\epi g^*$ is closed in
  the topology $\omega (X^*, X)\times \RR$.

\item
$(f+g)^*=f^*\Box g^*$  in $X^*$ and the infimal convolution is exact (attained) everywhere.

\item $\partial (f+g)=\partial f+\partial g$.

 \end{enumerate}
\end{corollary}

\begin{proof}
Apply Corollary~\ref{sublcSE:4} directly.
\end{proof}

We end this section with a corollary of our main result involving the subdifferential of the sum of convex functions. We recall that a formula known to hold in general, without any constraint qualification, has been given by Hiriart-Urruty and Phelps in \cite[Theorem 2.1]{HUP} (see also \cite[Corollary 5.1]{FS} and
\cite[Theorem 3.1]{HUMSV}) and is as follows.
\begin{equation}\label{HU-P}
 \partial(f_1+\cdots+f_m)(x)= \bigcap_{\eta>0} \overline{\left[ \partial_{\eta} f_1(x)+\cdots+\partial_{\eta} f_m(x)\right]}^{\wk}.
\end{equation}
Several constraint qualifications have been given in the literature to obtain simpler expressions for the right hand side in \eqref{HU-P}. As we mentioned before, the closed epigraph condition allows one to conclude the subdifferential sum formula, so both the intersection symbol and the closure operator become superfluous under this constraint qualification. Hence it is valid to ask whether our closedness condition in Theorem \ref{Theman:1}(i) allows us to simplify the right hand side in \eqref{HU-P}.
The following corollary shows that this is indeed the case, and we are able to remove the weak$^*$ closure from \eqref{HU-P}.

\begin{corollary}\label{CorsubDe:1}
Let $m\in\NN$, and $f_i:X\rightarrow\RX$ be proper convex with $\bigcap_{i=1}^{m}\dom f_i\neq\varnothing$,
 where $i\in\{1,2,\ldots,m\}$. Suppose  that $\overline {f_i}=f_i$  on $\bigcap_{i=1}^{m}\dom\overline {f_i}$.
Assuming any of the assumptions (i)-(iv) in Theorem~\ref{Theman:1}, the following equality holds for every $x\in X$,
  \begin{align*}
 \partial(f_1+\cdots+f_m)(x)= \bigcap_{\eta>0} \left[ \partial_{\eta} f_1(x)+\cdots+\partial_{\eta} f_m(x)\right].
 \end{align*}
\end{corollary}

\begin{proof}
By Theorem~\ref{Theman:1}(iv), we have
\begin{align*}
\partial(f_1+\cdots+f_m)(x)&= \bigcap_{\eta>0} \left[ \bigcup_{\varepsilon_i\geq0, \, \sum_{i=1}^m\varepsilon_i=\eta}
 \Big(\partial_{\varepsilon_1} f_1(x)+\cdots+\partial_{\varepsilon_m} f_m(x)\Big)\right]\\
 &\subseteq
\bigcap_{\eta>0}
\Big(\sum_{i=1}^m \partial_{\eta} f_i(x)\Big)\subseteq \bigcap_{\eta>0}\Big(\partial_{m\eta}(\sum_{i=1}^m f_i)(x)\Big)
=\partial(\sum_{i=1}^m f_i)(x).
\end{align*}
Hence $\partial(f_1+\cdots+f_m)(x)= \bigcap_{\eta>0} \left[ \partial_{\eta} f_1(x)+\cdots+\partial_{\eta} f_m(x)\right]$.
\end{proof}

Without the constraint qualification in Theorem~\ref{Theman:1}, Corollary~\ref{CorsubDe:1} need not hold, as shown in the following example.
We denote by $\overline{\spand\{C\}}$ the closed linear subspace spanned by a set $C$.

\begin{example}Let $\NN:=\{0,1,2,\ldots\}$.
Suppose that $H$ is an infinite-dimensional Hilbert space and let $(e_n)_{n\in\NN}$ be an orthonormal sequence in $H$. Set \begin{align*}
C:=\overline{\spand\{e_{2n}\}_{n\in\NN}}\quad\text{and}\quad  D:=\overline{\spand\{\cos(\theta_n) e_{2n}+\sin(\theta_n) e_{2n+1} \}_{n\in\NN}},
\end{align*}
where $(\theta_n)_{n\in\NN}$ is a sequence in $\left]0,\tfrac{\pi}{2}\right]$
such that $\sum_{n\in\NN} \sin^2(\theta_n)<+\infty$.  Define $f,g:H\rightarrow\RX$ by
\begin{align}
f:=\iota_{C^{\bot}}\quad\text{and}\quad g:=\iota_{D^{\bot}}.
\end{align}
Then $f$ and $g$ are proper lower semicontinuous and convex, and constraint qualifications in Theorem~\ref{Theman:1} fail. Moreover, \begin{align*}\partial(f+g)(x)\neq \bigcap_{\eta>0} \left[ \partial_{\eta} f(x)+\partial_{\eta} g(x)\right],\quad \forall x\in \dom f\cap\dom g.
\end{align*}
\end{example}
\begin{proof}
Since $C,D$ are closed linear subspaces, $f$ and $g$ are proper lower semicontinuous and convex.
Let $x\in \dom f\cap\dom g$ and $\eta>0$. Then we have $\partial_{\eta} f(x)=C^{\bot\bot}=C$ and
$\partial_{\eta} g(x)=D^{\bot\bot}=D$ and thus  $\partial_{\eta} f(x)+\partial_{\eta} g(x)=C+D$.
 Hence
 \begin{align}\bigcap_{\eta>0} \left[ \partial_{\eta} f(x)+\partial_{\eta} g(x)\right]=C+D.\label{CouExasu:1}
 \end{align}
 Then by \cite[Example~3.34]{BC2011}, $\bigcap_{\eta>0} \left[ \partial_{\eta} f(x)+\partial_{\eta} g(x)\right]$ is not norm closed and hence $\bigcap_{\eta>0} \left[ \partial_{\eta} f(x)+\partial_{\eta} g(x)\right]$ is not weak$^*$ closed by \cite[Theorem~3.32]{BC2011}.
However, $\partial(f+g)(x)$ is  weak$^*$ closed. Hence $\partial(f+g)(x)\neq \bigcap_{\eta>0} \left[ \partial_{\eta} f(x)+\partial_{\eta} g(x)\right]$.

Note that $\overline{\partial_{\eta} f(x)+\partial_{\eta} g(x)}^{\wk}=\overline{C+D}^{\wk}\nsubseteq C+D=\partial_{\varepsilon} f(x)+\partial_{\varepsilon} g(x),\,\forall \varepsilon>0$. Hence the constraint qualification in Theorem~\ref{Theman:1}\ref{mainGed:1} fails.
\end{proof}

\section{Further consequences of our main result}\label{intepCorL}
 In this section, we will recapture various forms of Rockafellar's Fenchel duality theorem.

\begin{lemma}[Interiority]
\label{DualCorL:Le4Re}
Let $m\in\NN$, and $\varepsilon_i\geq 0$ and let
 $f_i:X\rightarrow\RX$ be proper  convex,
 where $i\in\{1,2,\ldots,m\}$. Assume that there exists $x_0\in\big(\bigcap_{i=1}^{m}\dom f_i\big)$ such that
 $f_i$ is continuous at $x_0$ for every $i\in\{2,3,\ldots,m\}$.
 Then for every $x\in \big(\bigcap_{i=1}^{m}\dom f_i\big)$, the set $\sum_{i=1}^{m}\partial_{\varepsilon_i} f_i(x)$
is weak$^*$ closed. Moreover, for every $z\in \big(\bigcap_{i=1}^{m}\dom \overline{f_i}\big)$, the set $\sum_{i=1}^{m}\partial_{\varepsilon_i} \overline{f_i}(z)$
is weak$^*$ closed.
\end{lemma}

\begin{proof}
We can and do suppose that $x_0=0$.
Then  there exist a neighbourhood $V$ of $0$ and $K>\max\{0,f_1(0)\}$ such that  $V=-V$ (see \cite[Theorem~1.14(a)]{Rudin}) and
\begin{align}
V\subseteq\dom f_i\quad\text{and}\quad \sup_{y\in V}\overline{f_i}(y)\leq
\sup_{y\in V}f_i(y)\leq K,\quad\forall i\in\{2,3,\ldots,m\}.\label{GeDu:L1C0}
\end{align}
Let $x\in \bigcap_{i=1}^{m}\dom f_i$, $x^*\in \overline{\left[\sum_{i=1}^{m}\partial_{\varepsilon_i} f_i(x)\right]}^{\wk}$.  We will show that
\begin{align}
x^*\in\sum_{i=1}^{m}\partial_{\varepsilon_i} f_i(x).\label{GeDu:L1C1}
\end{align}
Our assumption on $x^*$ implies that for every $i=1,\ldots,m$ there exists a net
$(x^*_{i,\alpha})_{\alpha\in I}$ in $\partial_{\varepsilon_i} f_i(x)$
such that
\begin{align}
\sum_{i=1}^{m}x^*_{i,\alpha}\weakstarly x^*.\label{GeDu:L10C2}
\end{align}
We have
\begin{align}
 f^*_i(x^*_{i,\alpha})\leq-f_i(x)+\langle x, x^*_{i,\alpha}\rangle+\varepsilon_i,\quad
\forall i\in\{1,2,\ldots,m\}\quad\forall \alpha\in I\label{GeDu:L10C3}
\end{align}
Now we claim that
\begin{align}
\big\{\sum_{i=2}^m \sup|\langle x^*_{i,\alpha},V\rangle|\big\}_{\alpha\in I}=\big\{\sum_{i=2}^m \sup\langle x^*_{i,\alpha},V\rangle\big\}_{\alpha\in I}\quad\text{is eventually bounded}.\label{GeDu:L10C4a}
\end{align}
In other words, we will find a terminal set $J\subseteq I$ and $R>0$ such that
$\sum_{i=2}^m \sup\langle x^*_{i,\alpha},V\rangle\le R$ for all $\alpha\in J$. Fix $i\in\{2,\ldots,m\}$.  By \eqref{GeDu:L10C3}, we
have\begin{align}
&-f_i(x)+\langle x, x^*_{i,\alpha}\rangle+\varepsilon_i
\geq\sup_{y\in V}\{\langle x^*_{i,\alpha}, y\rangle-f_i(y)\}\geq\sup_{y\in V}\{\langle x^*_{i,\alpha}, y\rangle-K\}\quad\text{(by \eqref{GeDu:L1C0}})\nonumber\\
&=\sup\langle x^*_{i,\alpha},V\rangle-K.\label{GeDu:L10C4}
\end{align}
Then we have
\begin{align}
&-\sum_{i=2}^mf_i(x)+\langle x, \sum_{i=2}^m x^*_{i,\alpha}\rangle+\sum_{i=2}^m \varepsilon_i
\geq\sum_{i=2}^m\sup\langle x^*_{i,\alpha},V\rangle-(m-1)K,\quad\forall \alpha\in I.\label{GeDu:L10C5}
\end{align}
Since $0\in\dom f_1$ and , $f^*_1(x^*_{1,\alpha})\geq-f_1(0)\geq-K$. Then by
\eqref{GeDu:L10C3},
\begin{align}
&-f_1(x)+\langle x, x^*_{1,\alpha}\rangle+\varepsilon_1\geq-K,\quad\forall \alpha\in I.\label{GeDu:L10C6}
\end{align}
Combining \eqref{GeDu:L10C5} and \eqref{GeDu:L10C6}
\begin{align*}
-\sum_{i=1}^mf_i(x)+\langle x, \sum_{i=1}^mx^*_{i,\alpha}\rangle+\sum_{i=1}^m \varepsilon_i
\geq\sum_{i=2}^m\sup\langle x^*_{i,\alpha},V\rangle-mK,\quad\forall \alpha\in I.
\end{align*}
Then by \eqref{GeDu:L10C2},
\begin{align}
-\sum_{i=1}^mf_i(x)+\langle x,x^*\rangle+\sum_{i=1}^m \varepsilon_i
\geq\limsup_{\alpha \in I}\sum_{i=2}^m\sup\langle x^*_{i,\alpha},V\rangle-mK.\label{GeDu:L10C7}
\end{align}
Hence \eqref{GeDu:L10C4a} holds.

Then by \eqref{GeDu:L10C4a} and the Banach-Alaoglu Theorem
(see \cite[Theorem~3.15]{Rudin} or \cite[Theorem~1.1.10]{Zalinescu}),  there  exists a weak* convergent \emph{subnet}
$(x^*_{i,\gamma})_{\gamma\in\Gamma}$ of $(x^*_{i,\alpha})_{\alpha\in I}$ such that
\begin{align}x^*_{i,\gamma}\weakstarly  x^*_{i,\infty}\in X^*,\quad i\in\{2,\ldots,m\}.\label{GeDu:L10C8}\end{align}
Since $\partial_{\varepsilon_i}f_i(x)$ is weak$^*$ closed by \cite[Theorem~2.4.2]{Zalinescu}, then \begin{align}x_{i,\infty}^*
\in\partial_{\varepsilon_i}f_i(x),\quad\forall i\in\{2,\ldots,m\}.\label{GeDu:L10C9}
\end{align}
Then by \eqref{GeDu:L10C2},
\begin{align}
x^*-\sum_{i=2}^m x_{i,\infty}^*\in\partial_{\varepsilon_1}f_1(x).\label{GeDu:L10C10}
\end{align}
Combining the above two equations, we have
\begin{align*}
x^*\in\sum_{i=1}^m \partial_{\varepsilon_i}f_i(x).
\end{align*}
Hence
 $\sum_{i=1}^{m}\partial_{\varepsilon_i} f_i(x)$
is weak$^*$ closed.

Similarly, the set $\sum_{i=1}^{m}\partial_{\varepsilon_i} \overline{f_i}(z)$
is weak$^*$ closed  for every $z\in \big(\bigcap_{i=1}^{m}\dom \overline{f_i}\big)$.
\end{proof}

\begin{lemma}
\label{DualCorL:Le4}
Suppose that $X$ is a Banach space.
Let $m\in\NN$, and $\varepsilon_i\geq0$ and
 $f_i:X\rightarrow\RX$ be proper lower semicontinuous and convex,
 where $i\in\{1,2,\ldots,m\}$. Assume that \[\dom f_1\cap\big(\bigcap_{i=2}^{m}\inte\dom f_i\big)\neq\varnothing.\] Then for every $x\in \bigcap_{i=1}^{m}\dom f_i$, the set $\sum_{i=1}^{m}\partial_{\varepsilon_i} f_i(x)$
is weak$^*$ closed.
\end{lemma}
\begin{proof}
By
\cite[Proposition~3.3]{ph}, we conclude that $f_i$ is continuous for $i\in\{2,\ldots,m\}$. Apply now
Lemma~\ref{DualCorL:Le4Re} directly.
\end{proof}

The following results recapture various known exactness results as consequences of our main results.

\begin{corollary}
\label{DualCorL:RL4}\emph{(See \cite[Theorem~3.5.8]{BotGradWank}.)}
Let $m\in\NN$, and $\varepsilon_i\geq0$ and
 $f_i:X\rightarrow\RX$ be proper  convex,
 where $i\in\{1,2,\ldots,m\}$. Assume that  there exists $x_0\in\big(\bigcap_{i=1}^{m}\dom f_i\big)$ such that
 $f_i$ is continuous at $x_0$ for every $i\in\{2,3,\ldots,m\}$. Then
$(\sum_{i=1}^m f_i)^*=f_1^*\Box\cdots \Box f_m^*$  in $X^*$ and the infimal convolution is exact everywhere.
Furthermore, $\partial (f_1+f_2+\cdots+f_m)=\partial f_1+\cdots+\partial f_m$.
\end{corollary}

\begin{proof}
By \cite[Lemma~15]{BanLoZa},
\begin{align}
\overline{f_1+f_2\ldots+f_m}=\overline{f_1}+\overline{f_2\ldots+f_m}
=\ldots=\overline{f_1}+\overline{f_2}+\cdots+\overline{f_m}.\label{RoclGEx:1}
\end{align}
By the assumption, we have $x_0\in\dom \overline{f_1}\cap\big(\bigcap_{i=2}^{m}\inte\dom \overline{f_i}\big)$
and $\overline{f_i}$ is proper for every $i\in\{2,3,\ldots,m\}$ by \cite[Theorem~2.3.4(ii)]{Zalinescu}.

We consider two cases.

\emph{Case 1}: $\overline{f_1}$ is proper.

By \eqref{RoclGEx:1}, Lemma~\ref{DualCorL:Le4Re} and  Theorem~\ref{Theman:1} (applied to $\overline{f_i}$), we have
\begin{align}
(\sum_{i=1}^m f_i)^*=\Big(\overline{\sum_{i=1}^m f_i}\Big)^*=(\sum_{i=1}^m \overline{f_i})^*=\overline{f_1}^*\Box\cdots \Box \overline{f_m}^*=
f_1^*\Box\cdots \Box f_m^*. \label{GeDu:L10C10d}\end{align}
Let $x^*\in X^*$. Next we will show that
$(f_1^*\Box\cdots \Box f_m^*)(x^*)$ is achieved.
This is clear when $x^*\notin\dom (\sum_{i=1}^m f_i)^*$  by \eqref{GeDu:L10C10d}.
Now suppose that $x^*\in\dom (\sum_{i=1}^m f_i)^*$  and then
$(\sum_{i=1}^m f_i)^*(x^*)\in\RR$.
By \eqref{GeDu:L10C10d}, there exists $(x^*_{i,n})_{n\in\NN}$ such that $\sum_{i=1}^m x^*_{i,n}= x^*$
and
\begin{align}
f^*_1(x^*_{1,n})+f^*_2(x^*_{2,n})+\cdots+f^*_m(x^*_{m,n})
\leq(\sum_{i=1}^m f_i)^*(x^*)+\frac{1}{2n}.\label{GeDu:L10C11}
\end{align}
Since $x^*\in\dom (\sum_{i=1}^m f_i)^*$,
there exists $x\in X$ such that $x\in \partial _{\frac{1}{2n}} (\sum_{i=1}^m f_i)^* (x^*)$. Then
by \eqref{RoclGEx:1},
\begin{align*}
(\sum_{i=1}^m f_i)^*(x^*)+(\sum_{i=1}^m \overline{f_i})(x)&=
(\sum_{i=1}^m f_i)^*(x^*)+\big(\overline{\sum_{i=1}^m f_i}\big)(x)=
(\sum_{i=1}^m f_i)^*(x^*)+(\sum_{i=1}^m f_i)^{**}(x)\\
&\leq \langle x, x^*\rangle+\frac{1}{2n}.
\end{align*}
Then by \eqref{GeDu:L10C11},
\begin{align*}f^*_1(x^*_{1,n})+f^*_2(x^*_{2,n})+\cdots+f^*_m(x^*_{m,n})+(\sum_{i=1}^m \overline{f_i})(x)
\leq \langle x, x^*\rangle+\frac{1}{n}.
\end{align*}
Hence
\begin{align}
x^*_{i,n}\in\partial_{\frac{1}{n}} \overline{f_i}(x),\quad\forall i\in\{1,2,\ldots,m\},\forall n\in\NN.
\end{align}
By the assumptions,  there exist a neighbourhood $V$ of $0$ and $K>\max\{0,f_1(0)\}$ such that  $V=-V$ (see \cite[Theorem~1.14(a)]{Rudin}) and
\begin{align*}
V\subseteq\dom f_i\quad\text{and}\,
\sup\overline{f_i}(V)\leq\sup f_i(V)\leq K,\quad\forall i\in\{2,3,\ldots,m\}.
\end{align*}
As in the proof of Lemma~\ref{DualCorL:Le4Re},
$\big(\sum_{i=2}^m \sup|\langle x^*_{i,n}, V\rangle|\big)_{n\in\NN}$ is bounded and then
there  exists a weak* convergent \emph{subnet}
$(x^*_{i,\gamma})_{\gamma\in\Gamma}$ of $(x^*_{i,n})_{n\in\NN}$ such that
\begin{align}x^*_{i,\gamma}&\weakstarly  x^*_{i,\infty}\in X^*,\quad i\in\{2,\ldots,m\}\nonumber\\
x^*_{1,\gamma}&\weakstarly  x^*-\sum_{i=2}^mx^*_{i,\infty}\in X^*.\label{GeDu:L10C12}\end{align}
Combining \eqref{GeDu:L10C12} and taking the limit along the subnets in \eqref{GeDu:L10C11}, we have
\begin{align}
f^*_1( x^*-\sum_{i=2}^mx^*_{i,\infty})+f^*_2(x^*_{2,\infty})+\cdots+f^*_m(x^*_{m,\infty})
\leq(\sum_{i=1}^m f_i)^*(x^*).\label{GeDu:L10C14}
\end{align}
By
\eqref{GeDu:L10C10d} again and \eqref{GeDu:L10C14},
\begin{align*}
f^*_1( x^*-\sum_{i=2}^mx^*_{i,\infty})+f^*_2(x^*_{2,\infty})+\cdots+f^*_m(x^*_{m,\infty})
=(f_1^*\Box\cdots \Box f_m^*)(x^*).
\end{align*}
Hence $f_1^*\Box\cdots \Box f_m^*$ is achieved at $x^*$.

By Lemma~\ref{Lemresou},
we  have $\partial (f_1+f_2+\cdots+f_m)=\partial f_1+\cdots+\partial f_m$

\emph{Case 2}: $\overline{f_1}$ is  not proper.

Since $x_0\in\dom\overline{f_1}$, we have there exists $y_0\in X$ such that $\overline{f_1}(y_0)=-\infty$ and thus $ \overline{f_1}(x)=-\infty$ for every $x\in\dom \overline{f_1}$ by \cite[Proposition~2.4]{EkeTem}.
Thus by \eqref{RoclGEx:1},
\begin{align}
\overline{(f_1+f_2\ldots+f_m)}(x_0)=\overline{f_1}(x_0)+\overline{f_2}(x_0)+\cdots+\overline{f_m}(x_0)
=-\infty\label{RoeqTd:2}
\end{align} since $\overline{f_i}$ is proper for every $\in\{2,3,\ldots,m\}$ and $x_0\in\dom \overline{f_1}\cap\big(\bigcap_{i=2}^{m}\inte\dom \overline{f_i}\big)$.

We also have  $f^*_1=+\infty$ and then
\begin{align}
f_1^*\Box\cdots \Box f_m^*=+\infty.\label{RoeqTd:1}
\end{align}
Then by \eqref{RoeqTd:2}, we have
\begin{align*}
(\sum_{i=1}^m f_i)^*=\Big(\overline{\sum_{i=1}^m f_i}\Big)^*
=+\infty=f_1^*\Box\cdots \Box f_m^*.
\end{align*}
Hence $f_1^*\Box\cdots \Box f_m^*$ is exact everywhere.

Apply Lemma~\ref{Lemresou} directly to obtain that $\partial (f_1+f_2+\cdots+f_m)=\partial f_1+\cdots+\partial f_m$.

Combining the above two cases, the result holds.
\end{proof}

\begin{corollary}
\label{DualCorL:4}
Suppose that $X$ is a Banach space.
Let $m\in\NN$, and
 $f_i:X\rightarrow\RX$ be proper lower semicontinuous and convex with $\dom f_1\cap\big(\bigcap_{i=2}^{m}\inte\dom f_i\big)\neq\varnothing$, where $i\in\{1,2,\ldots,m\}$. Then
$(\sum_{i=1}^m f_i)^*=f_1^*\Box\cdots \Box f_m^*$  in $X^*$ and the infimal convolution is exact everywhere.
Furthermore, $\partial (f_1+f_2+\cdots+f_m)=\partial f_1+\cdots+\partial f_m$.
\end{corollary}

\begin{proof}
By
\cite[Proposition~3.3]{ph}, $f_i$ is continuous on $\inte\dom f_i$ for
$i\in\{2,\ldots,m\}$. Then apply Corollary~\ref{DualCorL:RL4} directly.
\end{proof}

\begin{corollary}[Rockafellar]
\label{DualCorL:5}\emph{(See {\cite[Theorem~4.1.19]{BorVan}} {\cite[Theorem~3]{Rock66}}, or
{\cite[Theorem~2.8.7(iii)]{Zalinescu}}.)}
Let
 $f,g:X\rightarrow\RX$ be proper  convex. Assume that there exists $x_0\in\dom f\cap\dom g$
 such that $f$ is continuous at $x_0$.
  Then
$(f+g)^*=f^*\Box g^*$ in $X^*$ and the infimal convolution is exact everywhere.
Furthermore, $\partial (f+g)=\partial f+\partial g$.

\end{corollary}
\begin{proof}
Apply Corollary~\ref{DualCorL:RL4} directly.
\end{proof}

A {\em polyhedral set} is a subset of a Banach space defined as a
finite intersection of halfspaces.  A function $f: X\rightarrow\RX$ is
said to be \emph{polyhedrally convex} if $\epi f$ is a polyhedral set.

\begin{corollary}\label{poly}
  Let $m,k, d\in\NN$ and suppose that $X=\RR^d$, let
  $f_i:X\rightarrow\RX$ be a polyhedrally convex function for
  $i\in\{1,2,\ldots,k\}$.  Let $f_j:X\rightarrow\RX$ be proper convex
  for every $j\in\{k+1,k+2,\ldots,m\}$.  Assume that there exists
  $x_0\in \bigcap_{i=1}^m\dom f_i$ such that $f_i$ is continuous at
  $x_0$ for every $i\in\{k+1,k+2,\ldots,m\}$.

 Then
$(\sum_{i=1}^m f_i)^*=f_1^*\Box\cdots \Box f_m^*$ in $X^*$ and the infimal convolution is exact everywhere.
Furthermore, $\partial (f_1+f_2+\cdots+f_m)=\partial f_1+\cdots+\partial f_m$.
\end{corollary}
\begin{proof}
Set $g_1:=\sum_{i=1}^k f_i$ and $g_2:=\sum_{i=k+1}^m f_i$.
By \cite[Corollary~19.1.2]{Rock70CA}, $f_i$ is lower semicontinuous for every $i\in\{1,2,\ldots,k\}$, so is $g_1$. By Corollary~\ref{DualCorL:5}, $(g_1+g_2)^*=g_1^*\Box g_2^*$ with the exact  infimal convolution
and $\partial (g_1+g_2)=\partial g_1 +\partial g_2$.

Let $i\in\{1,2,\ldots,k\}$. By \cite[Theorem~19.2]{Rock70CA}, $f^*_i$
is a polyhedrally convex function.  Hence $f_1^*\Box\cdots \Box f_m^*$
is polyhedrally convex by
\cite[Corollary~19.3.4]{Rock70CA} and hence $\sum_{i=1}^m \epi f_i^*$
is closed by \cite[Theorem~2.1.3(ix)]{Zalinescu} and \cite[Theorem~19.1]{Rock70CA}.  Then applying
Corollary~\ref{DualCorL:1}, we have $g_1^*=f_1^*\Box\cdots \Box f_k^*$
with the infimal convolution is exact everywhere. Using now Lemma~\ref{Lemresou} we obtain $\partial
g_1=\partial (f_1+f_2+\cdots+f_k)=\partial f_1+\cdots+\partial f_k$.

By Corollary~\ref{DualCorL:RL4}, we have $g_2^*=f_{k+1}^*\Box\cdots \Box
f_m^*$ with exact infimal convolution, and $\partial g_2=\partial
(f_{k+1}+f_{k+2}+\cdots+f_m)=\partial f_{k+1}+\cdots+\partial f_m$.

Combining the above results, we have $(\sum_{i=1}^m
f_i)^*=(g_1+g_2)^*=f_1^*\Box\cdots \Box f_m^*$ with  exact infimal
convolution, and $\partial (f_1+f_2+\cdots+f_m)=\partial
f_1+\cdots+\partial f_m$.
\end{proof}

\section{Conclusion}
We have introduced a new dual condition for zero duality gap in convex
programming. We have proved that our condition is less restrictive
than all other conditions in the literature, and we have related it
with (a) Bertsekas constraint qualification, (b) the closed epigraph
condition, and (c) the interiority conditions. We have used our
closedness condition to simplify the well-known expression for the
subdifferential of the sum of convex functions. Our study has
motivated the following open questions.
\begin{enumerate}
  \item Does the Closed Epigraph Condition imply Bertsekas Constraint Qualification?
  \item Are the conditions of Theorem \ref{Theman:1} strictly more restrictive than Bertsekas Constraint Qualification?
  \item How do these results extend when, instead of the sum of convex functions, the objective of the primal problem has the form $f+g\circ A$, where $f,g$ convex and $A$ a linear operator?
\end{enumerate}

\paragraph{Acknowledgments.} The authors thank the anonymous referee
for his/her pertinent and constructive comments. The authors are
grateful to Dr.\ Ern\"o Robert Csetnek for pointing out to us some
important references.  Jonathan Borwein and Liangjin Yao were
partially supported by the Australian Research Council.  The third
author thanks the School of Mathematics and Statistics (currently the
School of Information Technology and Mathematical Sciences) of
University of South Australia, for its support towards a visit to Adelaide,
which started this research.

\end{document}